\documentclass[a4paper,reqno,oneside,final]{amsart}

\usepackage{setspace}

\usepackage[notcite,notref,color]{showkeys} 
\definecolor{labelkey}{gray}{.9}

\usepackage{amsmath,amsfonts,amssymb,amsthm}
\usepackage{graphicx,psfrag}
\usepackage{color}
\usepackage[colorlinks,bookmarks,linkcolor=black,citecolor=black]{hyperref} 
\usepackage{enumitem}
\usepackage{verbatim,textcase}

\let\subset\subseteq

\def\subsc#1{\textsc{\MakeTextLowercase{#1}}} 
\def\sublem#1{\subsc{L\tiny{\ref{#1}}}}

\setlist{itemsep=1pt,parsep=0pt,topsep=2pt,partopsep=0pt}  
\setenumerate{leftmargin=*} 

\def\itm#1{\rm ({#1})} 
\def\itmit#1{\itm{\it #1\,}} 
\def\rom{\itmit{\roman{*}}} 
\def\abc{\itmit{\alph{*}}} 
 
\def\endofFact{\hfill\scalebox{.6}{$\Box$}}

\newcommand{\By}[2]{\overset{\mbox{\tiny{#1}}}{#2}} 
\newcommand{\ByRef}[2]{   \By{\eqref{#1}}{#2} }

\newcommand{\eqByRef}[1]{ \ByRef{#1}{=} } 
 
\newcommand{\gByRef}[1]{  \ByRef{#1}{>} } 
\newcommand{\leByRef}[1]{ \ByRef{#1}{\le} } 
\newcommand{\geByRef}[1]{ \ByRef{#1}{\ge} } 
\newcommand{\greaterByRef}[1]{ \ByRef{#1}{>} } 
 
\newcommand{\FF}{\mathbb{F}}

\newcommand{\bigO}{\mathcal{O}}

\newcommand{\eps}{\varepsilon}
\let\epsilon\varepsilon
\let\ln\log

\newcommand{\EK}{\tilde{K}}
\newcommand{\Good}{g}

\newcommand{\tpl}[1]{\boldsymbol{#1}}
\newcommand{\vecu}{\tpl{u}}
\newcommand{\vecv}{\tpl{v}}
\newcommand{\vecx}{\tpl{x}}
\newcommand{\vecy}{\tpl{y}}

\newcommand{\vecc}{\tpl{c}}

\newtheorem{theorem}{Theorem}
\newtheorem{lemma}[theorem]{Lemma}
\newtheorem{claim}[theorem]{Claim}
\newtheorem{proposition}[theorem]{Proposition}
\newtheorem{corollary}[theorem]{Corollary}

\newtheorem{conjecture}[theorem]{Conjecture}

\newtheorem{question}[theorem]{Question}

\theoremstyle{definition}
\newtheorem{definition}[theorem]{Definition}

\theoremstyle{remark}
\newtheorem{remark}[theorem]{Remark}

\newcommand{\oldqed}{}
\newenvironment{claimproof}[1][Proof]{
  \renewcommand{\oldqed}{\qedsymbol}
  \renewcommand{\qedsymbol}{\endofFact}
  \begin{proof}[#1]
}{
  \end{proof}
  \renewcommand{\qedsymbol}{\oldqed}
} 

\title{Powers of Hamilton cycles in pseudorandom graphs}

  \author[P. Allen]{Peter Allen}
  \author[J. B\"ottcher]{Julia B\"ottcher}
  \address{%
    Department of Mathematics, London School of Economics, Houghton Street,
    London WC2A 2AE, U.K.
  }
  \email{p.d.allen@lse.ac.uk, j.boettcher@lse.ac.uk}

  \author[H. H\`an]{Hi\d{\^{e}}p H\`an}
  \author[Y. Kohayakawa]{Yoshiharu Kohayakawa}
  \address{%
   Instituto de Matem\'atica e Estat\'{\i}stica, Universidade de
    S\~ao Paulo, Rua do Mat\~ao 1010, 05508--090~S\~ao Paulo, Brazil
  }
  \email{yoshi@ime.usp.br, hh@ime.usp.br}

  \author[Y. Person]{Yury Person}
  \address{%
    Goethe-Universit\"at, Institute of Mathematics,
    Robert-Mayer-Str. 10, 60325 Frankfurt, Germany 
  }
  \email{person@math.uni-frankfurt.de}

  \thanks{%
    PA was partially supported by FAPESP (Proc.~2010/09555-7),
    JB by FAPESP (Proc.~2009/17831-7),
    HH by FAPESP (Proc.~2010/16526-3).
    YK was partially supported by CNPq (308509/2007-2, 477203/2012-4),
    FAPESP (2013/03447-6, 2013/07699-0) and the NSF (DMS-1102086).
    The cooperation of the authors was
    supported by a joint CAPES/DAAD project (415/ppp-probral/po/D08/11629,
    Proj.~no.~333/09).
    The authors are grateful to NUMEC/USP, N\'ucleo de Modelagem Estoc\'astica e
    Complexidade of the University of S\~ao Paulo, and Project MaCLinC/USP, for
    supporting this research.
  }

\date{\today}

\begin{document}
\begin{abstract}
  We study the appearance of powers of Hamilton cycles in pseudorandom
  graphs, using the following comparatively weak pseudorandomness notion.
  A graph~$G$ is $(\eps,p,k,\ell)$-pseudorandom if for all disjoint
  $X$ and~$Y\subset V(G)$ with $|X|\ge\eps p^kn$ and $|Y|\ge\eps p^\ell n$ we have
  $e(X,Y)=(1\pm\epsilon)p|X||Y|$.  We prove that for all $\beta>0$
  there is an $\eps>0$ such that an 
  $(\eps,p,1,2)$-pseudorandom graph on $n$ vertices with minimum degree at
  least $\beta pn$ contains the square of a Hamilton cycle.  In particular,
  this implies that $(n,d,\lambda)$-graphs with $\lambda\ll d^{5/2 }n^{-3/2}$
  contain the square of a Hamilton cycle, and thus a triangle factor if~$n$
  is a multiple of~$3$. This improves on a result of Krivelevich, Sudakov
  and Szab\'o [\emph{Triangle factors in sparse pseudo-random graphs},
  Combinatorica \textbf{24} (2004), no.~3, 403--426].

  We also extend our result to higher powers of Hamilton cycles and
  establish corresponding counting versions.
\end{abstract}

\maketitle

\section{Introduction and results}

The appearance of certain graphs~$H$ as subgraphs is a dominant topic
in the study of random graphs. In the random graph model~$G(n,p)$ this
question turned out to be comparatively easy for graphs~$H$ of
constant size, but much harder for graphs~$H$ on~$n$ vertices, i.e.,
\textit{spanning} subgraphs. Early results were however obtained in
the case when~$H$ is a Hamilton cycle, for which this question is by
now very well understood~\cite{Boll,KomSzem,Kor76,Kor77,Posa76}.

When we turn to other spanning subgraphs~$H$ rather little was known
for a long time, until a remarkably general result by
Riordan~\cite{Riordan} established good estimates for a big variety of
spanning graphs~$H$.  In particular his result determines the
threshold for the appearance of a spanning hypercube, and the
threshold for the appearance of a spanning square lattice, as well as
of the $k$th-power of a Hamilton cycle for $k>2$. Here the \emph{$k$th
  power} of $H$ is obtained from~$H$ by adding all edges between
distinct vertices of distance at most~$k$ in~$H$.  For the square of a
Hamilton cycle the corresponding approximate threshold was only obtained recently by
K\"uhn and Osthus~\cite{KOPosa}.

Observe that the $k$th power of a Hamilton cycle contains $\lfloor
n/(k+1)\rfloor$ vertex disjoint copies of $K_{k+1}$, a so-called
\emph{$K_{k+1}$-factor}. It came as another breakthrough in the
area and solved a long-standing problem when Johansson, Kahn and
Vu~\cite{JKV08} established the threshold for $K_{k+1}$-factors in
$G(n,p)$ (or more generally of certain $F$-factors).

\subsection{Pseudorandom graphs}

Thomason~\cite{Tho87} asked whether it is possible to single out some
properties enjoyed by~$G(n,p)$ with high probability that
deterministically imply a rich collection of structural results that
hold for~$G(n,p)$.  He thus initiated the study of pseudorandom graphs
and suggested a deterministic property similar to the following notion
of jumbledness.  An $n$-vertex graph $G$ is \emph{$(p,\beta)$-jumbled}
if
\begin{equation}\label{eq:mixing}
  \big|e(A,B)-p|A||B|\big|\le \beta \sqrt{|A||B|}
\end{equation}
for all disjoint $A,B\subset V(G)$, where~$e(A,B)$ is the number of
edges in~$G$ with one endvertex in~$A$ and the other endvertex in~$B$.
The random graph $G(n,p)$ is with high probability $(p,\beta)$-jumbled
with $\beta=O(\sqrt{pn})$, so this definition is justified.  Moreover,
this pseudorandomness notion indeed implies a rich structure (see,
e.g., \cite{ChuGraWil,ChuGra,ConFoxZha,Tho87}). However, for spanning
subgraphs of general jumbled graphs (with a suitable minimum degree
condition) not much is known.

One special class of jumbled graphs, which has been studied extensively, is
the class of $(n,d,\lambda)$-graphs. Its definition relies on spectral properties.
For a graph~$G$ with eigenvalues $\lambda_1\geq\lambda_2\geq\dots\geq
\lambda_n$ of the adjacency matrix of~$G$, we call
$\lambda(G):=\max\{|\lambda_2|,|\lambda_n|\}$ the \emph{second eigenvalue}
of~$G$.  An \emph{$(n,d,\lambda)$-graph} is a $d$-regular graph on $n$
vertices with $\lambda(G)\leq \lambda$.  The connection between
$(n,d,\lambda)$-graphs and jumbled graphs is established by the well-known
\emph{expander mixing lemma} (see, e.g., \cite{AS00}), which states that if
$G$ is an $(n,d,\lambda)$-graph, then
\begin{equation}\label{eq:nd_lambda}
  \left|e(A,B)-\tfrac{d}{n}|A||B|\right|\le \lambda(G) \sqrt{|A||B|}
\end{equation}
for all disjoint subsets $A,B\subset V(G)$. Hence~$G$ is $\big(\frac
dn,\lambda(G)\big)$-jumbled. 

One main advantage of $(n,d,\lambda)$-graphs are the powerful tools from
spectral graph theory which can be used for their study. Thanks to these
tools various results concerning spanning subgraphs of
$(n,d,\lambda)$-graphs~$G$ have been obtained. It turns out that already an
almost trivial eigenvalue gap guarantees a spanning matching: if
$\lambda\le d-2$ and~$n$ is even, then~$G$ has a perfect
matching~\cite{KriSudSurvey}. Moreover, if $\lambda\le d(\log\log 
n)^2/ (1000\log n \log\log\log n)$ then~$G$ has a Hamilton
cycle~\cite{KriSudHam}. The only other embedding result for spanning
subgraphs of $(n,d,\lambda)$-graphs that we are aware of concerns triangle
factors.  Krivelevich, Sudakov and Szab\'o~\cite{KSS04} proved that an
$(n,d,\lambda)$-graph~$G$ with $3\mid n$ and $\lambda=o\big(d^3/n^2\log n\big)$
contains a triangle factor.

It is instructive to compare this last result with corresponding lower bound
constructions.  Krivelevich, Sudakov and Szab\'o also remarked that by
using a blow-up of a construction of Alon~\cite{Alon94} one can obtain for
each $d'=d'(n')$ with $\Omega\big((n')^{2/3}\big)=d'\le n'$ an
$(n,d,\lambda)$-graph with $n=\Theta(n')$, $d=\Theta(d')$ and
$\lambda=\Theta(d^2/n)$ which is triangle-free and thus contains no
triangle factor. They conjectured that in fact
$(n,d,\lambda)$-graphs are so symmetric that the upper bound on~$\lambda$
they proved for triangle factors can be
improved, possibly all the way down to this lower bound. In this paper we bring
the upper bound closer to the conjectured lower bound and establish more
generally an embedding result for $k$th powers of Hamilton cycles (see
Corollary~\ref{cor:ndlambda}).

\subsection{Our results}

The pseudorandomness notion we shall work with in this paper is weaker than
that of $(n,d,\lambda)$-graphs, and in fact even weaker than jumbledness.

\begin{definition}
  Suppose~$\eps>0$ and~$0<p<1$.  Let $k$ and~$\ell$ with $k\le \ell $
  be positive integers.  We call an $n$-vertex graph~$G$
  $(\eps,p,k,\ell)$-\emph{pseudorandom} if
  \begin{equation}\label{eq:pseudodisc}
    \bigl|e(X,Y)-p|X||Y|\bigr|< \eps p |X||Y|
  \end{equation}
  for any disjoint subsets $X$, $Y\subseteq V(G)$ with
  $|X|\ge\eps p^{k}n$ and $|Y|\ge\eps p^{\ell}n$.
\end{definition}

It is easy to check that a graph which is
$\big(p,\eps^2p^sn\big)$-jumbled is
$(\eps,p,k,\ell)$-pseu\-do\-ran\-dom for all $k$ and~$\ell$ with
$k+\ell=2s-2$, but the jumbledness condition imposes tighter control
on the edge density between (for example) linear sized subsets. An
easy application of Chernoff's inequality and the union bound show
that $G(n,p)$ is $(\eps,p,k,\ell)$-pseudorandom with high probability
if $p\gg (n^{-1}\log n)^{1/(\max\{k,\ell\}+1)}$, while $G(n,p)$ only
gets $\big(p,\eps^2p^{(k+\ell+2)/2}n\big)$-jumbled if $p\gg 
n^{-1/(k+\ell+1)}$.  Our major motivation for using this weaker
pseudorandomness condition is that it is all we require.

Our main result states that sufficiently pseudorandom graphs which also
satisfy a mild minimum degree condition contain spanning powers of Hamilton
cycles.

\begin{theorem}\label{thm:main}
  For every $k\ge 2$ and $\beta>0$ there is an $\eps>0$ such that for any
  $p=p(n)$ with $0<p<1$ the following holds. Let~$G$ be a graph on~$n$
  vertices with minimum degree $\delta(G)\geq \beta pn$. 
  \begin{enumerate}[label=\abc]
    \item If~$G$ is $(\eps,p,1,2)$-pseudorandom
      then~$G$ contains a square of a Hamilton cycle.
    \item If~$G$ is $(\eps,p,k-1,2k-1)$-pseudorandom and $(\eps,p,k,k+1)$-pseudorandom
      then~$G$ contains a $k$th power of a Hamilton cycle.
  \end{enumerate}
\end{theorem}

We remark that our proof of Theorem~\ref{thm:main} also yields a
deterministic polynomial time algorithm for finding a copy of the $k$th
power of the Hamilton cycle. The proof technique (see
Section~\ref{subsec:outline} for an overview) is partly inspired by the
methods used in~\cite{TightCycle} (which have similarities to those of
K\"uhn and Osthus~\cite{KOPosa}).

It is immediate from the discussion above that our theorem implies the
following result for jumbled graphs.

\begin{corollary}[Powers of Hamilton cycles in jumbled graphs]
  For every $k\ge 2$ and $\beta>0$ there is an $\eps>0$ such that for any
  $p=p(n)$ with $0<p<1$ the following holds. Let~$G$ be a graph on~$n$
  vertices with minimum degree $\delta(G)\geq \beta pn$.  
  \begin{enumerate}[label=\abc]
    \item If~$G$ is $(p,\eps p^{5/2}n)$-jumbled  then~$G$
      contains a square of a Hamilton cycle.
    \item If~$G$ is $(p,\eps p^{3k/2}n)$-jumbled  then~$G$
      contains a $k$th power of a Hamilton cycle.
  \end{enumerate}
\end{corollary}

As a consequence we also obtain a corresponding corollary for $(n,d,\lambda)$-graphs.

\begin{corollary}[Powers of Hamilton cycles in
  $(n,d,\lambda)$-graphs]\label{cor:ndlambda}
  For all $k\ge2$ there is $\eps>0$ such that for every
  $(n,d,\lambda)$-graph $G$,
  \begin{enumerate}[label=\abc]
  \item if $\lambda\le\eps d^{5/2}n^{-3/2}$ then~$G$
    contains a square of a Hamilton cycle,
  \item if $\lambda\le\eps d^{3k/2}n^{1-3k/2}$
    then~$G$ contains a $k$th power of a Hamilton cycle.
  \end{enumerate}
\end{corollary}

In particular, under the conditions above, the graph~$G$ contains a
spanning triangle factor and a spanning $K_{k+1}$-factor, respectively,
if~$3\mid n$ and~$(k+1)\mid n$.  Thus we improve on the result of
Krivelevich, Sudakov and Szab\'o~\cite{KSS04} for triangle factors and
extend it to $K_{k+1}$-factors.

As remarked above even for $k=2$ our upper bound for $\lambda$ does not
match the known lower bound. For $k>2$ the situation gets even more
complicated since `good' lower bounds for the appearance of
$K_{k+1}$ (let alone $k$th powers of Hamilton cycles) in
$(n,d,\lambda)$-graphs are not available. The best we can do is to observe
that $G(n,p)$ with $(\ln n/n)^{1/(k-\eps)}\ll p\ll n^{-1/k}$ almost surely has no $k$th power of a
Hamilton cycle, and that such a graph for any fixed $\eps>0$ is almost surely
$(\eps,p,k-1-\eps,k-1-\eps)$-pseudorandom.

\subsection{Counting} 

Closely related to the question of the appearance of a certain subgraph in
random or pseudorandom graphs is the question of how many copies of this
subgraph are actually present. Janson~\cite{Jan94} and Cooper and Frieze~\cite{CooFri}
studied this problem for Hamilton cycles in
$G(n,p)$. Motivated by these results Krivelevich~\cite{KrivCount} recently
turned to counting Hamilton cycles in
sparse $(n,d,\lambda)$-graphs~$G$. He showed
that for every $\eps>0$ and sufficiently large~$n$, if $\lambda\le d/(\log
n)^{1+\eps}$ and $\log\lambda\ll\log d-\log n/\log d$ then~$G$ contains 
$n!(d/n)^n\big(1+o(1)\big)^n$ Hamilton cycles. This count is close to the
expected number of labeled Hamilton cycles in $G(n,p)$ with $p=d/n$, which is
$n!(d/n)^n$.

Krivelevich remarked that jumbled graphs may have isolated vertices and
thus no Hamilton cycles at all. The same applies to our notion of
pseudorandomness. If however, as in our main result, we combine this
pseudorandomness with a minimum degree condition to avoid this obstacle, we
do obtain a corresponding result concerning the number of Hamilton cycle
powers in such graphs. Again, we obtain a count close to $p^{kn}n!$, which
is the expected number of labeled copies of the $k$th power of a Hamilton
cycle in $G(n,p)$. Note that (unlike Krivelevich) we do not provide a
corresponding upper bound.

\begin{theorem}\label{thm:countHC}
  For every $k\ge 2$, $\beta$ and~$\nu>0$ there is a constant~$c>0$, such
  that for every $\eps=\eps(n)\le c/\log^2 n$ and $p=p(n)$ with
  $0<p<1$ the following holds.  Let~$G$ be a graph on~$n$ vertices
  with minimum degree $\delta(G)\geq \beta pn$.  Suppose that~$G$ is
  $(\eps,p,1,2)$-pseudorandom if $k=2$, and
  $(\eps,p,k-1,2k-1)$-pseudorandom and $(\eps,p,k,k+1)$-pseudorandom
  if $k>2$.  Then~$G$ contains  at least $(1-\nu)^np^{kn}n!$ copies of the $k$th
  power of a Hamilton cycle.
\end{theorem}

With some minor modifications, this result follows from our proof of
Theorem~\ref{thm:main}. For the sake of clarity, we sketch these
modifications after detailing the proof of Theorem~\ref{thm:main}.

\subsection{Organisation}

The remainder of this paper is organised as follows. In
Section~\ref{sec:main} we give some basic definitions, outline our proof
strategy, provide the main lemmas and use them to obtain
Theorem~\ref{thm:main}. In Sections~\ref{sec:connection}
and~\ref{sec:reservoir} we prove our three main lemmas. We sketch how to
modify the proof of Theorem~\ref{thm:main} to get
Theorem~\ref{thm:countHC} in Section~\ref{sec:count}, and close with some
remarks and open problems in Section~\ref{sec:conc}.

\section{Main lemmas and proof of the main theorem}
\label{sec:main}

\subsection{Notation}

An $s$-tuple $(u_1,\ldots,u_s)$ of vertices is an ordered set of vertices. We
often denote tuples by bold symbols, and occasionally also omit the brackets
and write $\vecu=u_1,\ldots,u_s$. 

Given a graph $H$, the graph $H^k$, called the $k$th \textit{power} of
$H$, is the graph on $V(H)$ where two distinct vertices $u$ and $v$
are adjacent if and only if their distance in $H$ is at most
$k$.

For simplicity we also call the $k$th power of a path a
\emph{$k$-path}, and the $k$th power of a cycle a \emph{$k$-cycle}.
We will usually specify $k$-paths and $k$-cycles by giving the
(cyclic) ordering of the vertices in the form of a vertex tuple.  We
say that the \emph{start} $s$-tuple of a $k$-path
$P=(u_1,\ldots,u_\ell)$ is $(u_s,\ldots,u_1)$, and the \emph{end}
$s$-tuple is $(u_{\ell-s+1},\ldots,u_\ell)$ (the
vertices~$u_{s+1},\dots,u_{\ell-s}$ are said to be internal).  In
these definitions, we shall often have~$s=k$.

For a given graph $G$ let $N_{X}(x)$ be the set of neighbours of $x$ in $X\subset V(G)$. For
an $\ell$-tuple $\vecx_{\ell}=(x_1,\ldots,x_{\ell})$ of
vertices let $N_X(x_1,\dots,x_{\ell})$ denote the common 
neighbourhood of $x_1,\dots,x_{\ell}$ in $X$, and let
$\deg_X(x_1,\ldots,x_{\ell})=|N_X(x_1,\dots, x_{\ell})|$.

We say that $\vecx_{\ell}$ is $(\varrho,p)$\emph{-connected} to a vertex set
$X$
if $x_1,\ldots, x_{\ell}$ forms a clique in $G$ and  
\begin{equation}\label{eq:connected}
 \deg_{X}(x_i,\dots,x_{\ell})\ge \varrho
 \left(\frac{p}{2}\right)^{\ell-i+1}|X| 
\end{equation}
for every $i\in[\ell]=\{1,\dots,\ell\}$.
To motivate this definition, note that the bound in~\eqref{eq:connected}
corresponds to the expected number of common neighbours of
$(x_i,\dots,x_\ell)$ in $X$ in the random graph $G(n,p)$, up to a constant
factor.

A vertex set~$Y\subset X$ \emph{witnesses} that~$\vecx_{\ell}$ is
$(\varrho,p)$-connected to~$X$ if for every $i\in[\ell]$ we have
$\big| Y\cap N_X(x_i,\dots,x_{\ell}) \big|\ge\varrho
\left(\frac{p}{2}\right)^{\ell-i+1} |X|$.

\begin{remark}\label{rem:witness}
  Since the sets $N_X(x_1,\dots,x_{\ell})$, $N_X(x_2,\dots,x_{\ell})$,
  \dots, $N_X(x_{\ell})$ are nested we have that if $\vecx_{\ell}$ is
  $(\varrho,p)$-connected to~$X$, then there is a set $Y\subset X$ with
  $|Y|=\varrho\frac{p}{2}|X|$ vertices which witnesses this connectedness.
\end{remark}

In our proofs we shall additionally frequently make use of the following observation
concerning our pseudorandomness notion.

\begin{remark}\label{rem:pseudo}
  If $0<p\le 1/2$ and $\eps<1/8$, and the $n$-vertex graph $G$ is
  $(\eps,p,k,\ell)$-pseudorandom, then $G$ has a vertex $y$ of degree at
  most $3n/4$. Furthermore, letting $X=V(G)\setminus\big(\{y\}\cup N(y)\big)$ and
  $Y=\{y\}$ we see that the pseudorandomness condition~\eqref{eq:pseudodisc} does not hold. It follows
  that $1<\eps p^\ell n$, or equivalently $p^\ell n>\eps^{-1}$. A similar
  statement holds if $1/2\le p<1$, taking $X=N(y)$. Thus assuming the
  $n$-vertex graph $G$ to be $(\eps,p,k,\ell)$-pseudorandom for any $0<p<1$
  implicitly means we assume $p^\ell n>\eps^{-1}$.
\end{remark}

\subsection{Outline of the proof}
\label{subsec:outline}

Suppose that $G$ is an $(\eps,p,k-1,k)$-pseudorandom graph on $n$ vertices.
One crucial observation, which forms the starting point of our proof, is
that it is relatively easy to find an almost spanning $k$-path
in~$G$. Indeed, it is not hard to check (see the Extension lemma,
Lemma~\ref{lem:onestep}) that~$G$ contains copies of~$K_k$ and that
typically such a $K_k$-copy is well-connected to the rest of the graph in
the following sense. There are many vertices which extend this $K_k$-copy
to a $k$-path on $k+1$ vertices. Iterating this argument we can greedily
build a $k$-path~$P'$ covering most of~$G$. Let~$L$ be the set of leftover
vertices.

Thus, the true challenge is to incorporate the few remaining vertices
into~$P'$ and to close~$P'$ into a $k$-cycle. To tackle the second of these
tasks we will establish a Connection lemma (Lemma~\ref{lem:connect}), which
asserts that any two pairs of $k$-cliques in~$G$ which are sufficiently
well-connected to a set~$U$ of vertices can be connected by a
short $k$-path with interior vertices in~$U$.  At this point, if~$k>2$, we shall need to require
that~$G$ be $(\eps,p,k-1,2k-1)$-pseudorandom.

For the first task, we make use of the \emph{reservoir method} developed
in~\cite{TightCycle} (see also~\cite{KOPosa} for a similar method). In
essence, the fundamental idea of this method is to ensure that~$P'$
contains a sufficiently big proportion of vertices which are free to be
taken out of~$P'$ and used otherwise.  More precisely, we shall construct
(see the Reservoir lemma, Lemma~\ref{lem:reslem}) a path $P$ with the
\emph{reservoir property}: There is a subset $R$ of $V(P)$, called the
\emph{reservoir}, such that for any $W\subset R$ there is a $k$-path in $G$
whose vertex set is $V(P)\setminus W$ and whose ends are the same as those
of~$P$. We also call~$P$ a \emph{reservoir path}.  We then use the greedy method
outlined above to extend~$P$ to an almost spanning $k$-path~$P'$.  For this step,
if~$k>2$, we shall need to require that~$G$ be
$(\eps,p,k,k+1)$-pseudorandom.

With the reservoir property we are now in good shape to incorporate the
leftover vertices~$L$ into~$P'$ (and then close the path into a cycle): We
show, using a Covering lemma (Lemma~\ref{lem:covlem}), that we can find a
$k$-path $P''$ in $L\cup R$ covering all vertices of~$L$ and using only a
small fraction of~$R$ (this is possible because~$R$ is much bigger
than~$L$).  Finally we connect both ends of $P'$ and $P''$ using some of
the remaining vertices of~$R$ with the help of the Connection lemma (again,
this is possible because many vertices of~$R$ remain).

Now the only problem is that some vertices of~$R$ may be used twice, in~$P'$
and in~$P''$ or the connections. But this is where the reservoir property
comes into play. This property asserts that there is a $k$-path $\widetilde P$
which uses all vertices of~$P'$ except these vertices. Finally $\widetilde P$
and $P''$ together with the connections form the desired spanning $k$-cycle.

\subsection{Main lemmas}

The proof of Theorem~\ref{thm:main} relies on four main lemmas, the
Extension lemma, the Reservoir lemma, the Covering lemma and the Connection
lemma, which we will state and explain in the following.

Our first lemma, the Extension lemma, states that in a sufficiently
pseudorandom graph all well-connected $k$-tuples have a common neighbour
which together with the last $k-1$ vertices of this $k$-tuple form again a
well-connected $k$-tuple.

\begin{lemma}[Extension lemma]\label{lem:onestep}
Given $k\geq 2$ and $\delta>0$ there is an $\eps>0$ such that
for all $0<p<1$, all 
$(\eps,p,k-1,k)$-pseudorandom graphs~$G$ on $n$ vertices, 
and all disjoint vertex sets $L$ and $R$ with $|L|,|R|\geq \delta n$ the
following holds.
 
Let $\vecx=(x_1,\dots,x_k)$ be a $k$-tuple which is
$(\tfrac18,p)$-connected to both~$L$ and~$R$.  Then there is a vertex
$x_{k+1}$ of $L\cap N(x_1,\ldots,x_k)$ such that
$(x_2,\ldots,x_{k+1})$ is $(\tfrac16,p)$-connected to both~$L$
and~$R$.
\end{lemma}

We stress that in this lemma we require and obtain well-connectedness to two
sets $L$ and $R$. This will enable us in the proof of Theorem~\ref{thm:main}
to extend a $k$-path alternatively using vertices of the leftover set $L$
or the reservoir set $R$.

We remark moreover that the assumed $(\tfrac18,p)$-connectedness is weaker
than the $(\tfrac16,p)$-connectedness in the conclusion.  This is useful
when we repeatedly apply the Extension lemma.  It is possible to prove such
a statement because the factor~$\frac12$ in the definition of connectedness
allows for some leeway.

Since the proof of this lemma is short we give it straight away. We use the
following lemma, which is a direct consequence of~\eqref{eq:pseudodisc} and
will frequently be used later as well.

\begin{lemma}\label{lem:pseudorandom}
  If~$G$ is an $(\eps,p,k,\ell)$-pseudorandom graph on~$n$ vertices and
  $X\subset V(G)$ satisfies $|X|\ge\eps p^kn$, then less than
  $\eps p^\ell n$ vertices $v\in V(G)\setminus X$ have
  $\deg_X(v)<(1-\eps)p|X|$, and  less than
  $\eps p^\ell n$ vertices $v\in V(G)\setminus X$ have
  $\deg_X(v)>(1+\eps)p|X|$.
\qed
\end{lemma}

\begin{proof}[Proof of Lemma~\ref{lem:onestep}]
 Given $k$ and $\delta$ we set
 \begin{equation}\label{eq:onestep:eps}
    \eps=\frac{\delta}{80\cdot k\cdot 2^{k+3}}\,.
  \end{equation}
 Because $\vecx$ is $(\tfrac18,p)$-connected to $L$, for each
$2\le i\le k$ we have
\begin{equation*}
 \deg_{L}(x_i,\dots,x_{k})\ge \frac{1}{8}
\left(\frac{p}{2}\right)^{k-i+1}
|L|\,.
\end{equation*}
We claim that for each $2\le i\le k$ there are less than $\eps p^k n$
vertices which have less than $\tfrac16(p/2)^{k-i+2}|L|$ neighbours in
$N(x_i,\ldots,x_k)\cap L=:Y_i$.  Indeed, we have $|Y_i|\ge
\frac18(\frac{p}{2})^{k-i+1}|L|\ge \frac18(\frac{p}{2})^{k-i+1}\delta
n>10\eps p^{k-1}n$.  Now assume for contradiction that there is a set~$B_i$
of $\eps p^k n$ vertices in~$V(G)$ all of which have less than 
$\tfrac16(p/2)^{k-i+2}|L|\le \frac23p|Y_i|$ neighbours in~$Y_i$. Since
$|B_i|\le\frac{1}{10}|Y_i|$ and thus $|Y_i\setminus
B_i|\ge\frac9{10}|Y_i|>\eps p^{k-1} n$, this implies that each vertex
in~$B_i$ has less than $\frac23p|Y_i|\le \frac23p \frac{10}{9}|Y_i\setminus
B_i|=\frac{20}{27} p|Y_i\setminus B_i|$ neighbours in $Y_i\setminus
B_i$. This however contradicts Lemma~\ref{lem:pseudorandom} because~$G$ is
$(\eps,p,k-1,k)$-pseudorandom.

Similarly, less than $\eps p^k n$ vertices have fewer than
$\tfrac16(p/2)|L|$ neighbours in $L$. The same calculations, replacing $L$
with $R$, also hold. It follows that all but at most $2k\eps p^k n$
vertices $x_{k+1}$ of $N(x_1,\ldots,x_k)\cap L$ have the property
that $(x_2,\ldots,x_{k+1})$ is $(\tfrac16,p)$-connected to both $L$ and
$R$. Finally, since
\[\deg_L(x_1,\ldots,x_k)\ge
\frac18\Big(\frac{p}{2}\Big)^k|L|\ge\frac{\delta p^k n}{
2^{k+3}}\gByRef{eq:onestep:eps}2k\eps p^kn\]
there is indeed a vertex $x_{k+1}$ with this property as desired.
\end{proof}

Our second lemma allows us to construct the reservoir path~$P$
described in the outline, given a suitable reservoir $R$ (see
properties~\ref{reslem:R} and~\ref{reslem:W} of the lemma). In addition, this lemma
guarantees well-connectedness of the ends of this path to the reservoir and
to the remaining vertices in the graph (see properties~\ref{reslem:connV}
and~\ref{reslem:connR} of the lemma). This is necessary so that we can extend the
reservoir path and later connect it to the path covering the leftover
vertices~$L$ using~$R$.

\begin{lemma}[Reservoir lemma]\label{lem:reslem}
  Given $k\ge 2$, $0<\delta< 1/4$ and $0<\beta<1/2$ there exists an $\eps>0$
  such that the following holds.

  Let $0<p<1$ and let $G=(V,E)$ be an $n$-vertex graph.  Suppose that $G$
  is $(\eps,p,1,2)$-pseudorandom if $k=2$, and
  $(\eps,p,k-1,2k-1)$-pseudorandom and $(\eps,p,k,k+1)$-pseudorandom if
  $k>2$.  Let $R\subset V$ satisfy $\delta^2 n/(200k)\le|R|\le\delta
  n/(200k)$ and $\deg_{V\setminus R}(v)\ge \beta pn/2$ for all $v\in
  R$.
  Then there is a $k$-path~$P$ in~$G$ with  the following properties.
  \begin{enumerate}[label=\abc]
  \item\label{reslem:R} $R\subset V(P)$, $|V(P)|\le 50k|R|$, and all
    vertices from~$R$ are internal in~$P$. 
  \item \label{reslem:connV}  The start and end $k$-tuples of $P$ are
    $(\tfrac18,p)$-connected to $V\setminus V(P)$.
  \item\label{reslem:connR} The start and end $k$-tuples of $P$ are
    $(\tfrac12,p)$-connected to $R$  (and thus disjoint from $R$).
  \item\label{reslem:W} For any $W\subset R$, there is a $k$-path with the vertex set
    $V(P)\setminus W$ whose start and end $k$-tuples are identical to those of $P$. 
  \end{enumerate}
\end{lemma}

Our third lemma enables us to cover the leftover vertices~$L$ with a
$k$-path (see property~\ref{covlem:L}). This lemma allows us in addition to
specify a set~$S$ to which the start and end tuples of this path have to
maintain well-connectedness (see property~\ref{covlem:S}).  When we cover
the leftover vertices in the proof of the main theorem, $S$ will be a big
proportion of~$R$ and we will use the well-connectedness to connect the
path covering~$L$ and the extended reservoir path.

Observe that the requirements and conclusions of Lemma~\ref{lem:reslem} and
Lemma~\ref{lem:covlem} overlap substantially. In fact, we shall prove both
lemmas together in Section~\ref{sec:reservoir}.

\begin{lemma}[Covering lemma]\label{lem:covlem}
  Given $k\ge 2$, $0<\delta<1/4$ and $0<\beta<1/2$, there exists an
  $\eps>0$
  such that the following holds.

  Let $0<p<1$ and let $G=(V,E)$ be an $n$-vertex graph.  Suppose that $G$
  is $(\eps,p,1,2)$-pseudorandom if $k=2$, and
  $(\eps,p,k-1,2k-1)$-pseudorandom and $(\eps,p,k,k+1)$-pseudorandom if
  $k>2$.  Let~$L$ and~$S$ be disjoint subsets of $V(G)$ with $|L|\le\delta
  n/(200k)$ and $|S|\ge\delta n$ such that $\deg_{S}(v)\ge \beta\delta
  pn/2$ for all $v\in L$.
  Then there is a $k$-path $P$ contained in $L\cup S$ with the following
  properties.
  \begin{enumerate}[label=\abc]
  \item\label{covlem:L} $L\subset V(P)$ and $|V(P)|\le 50k|L|$.
  \item \label{covlem:S}  The start and end $k$-tuples of $P$ are in $S$ and are
    $(\tfrac18,p)$-connected to $S\setminus V(P)$.
  \end{enumerate}
\end{lemma}

Our fourth and final main lemma allows us to connect two $k$-tuples with a
short $k$-path. 

\begin{lemma}[Connection lemma]\label{lem:connect}
  For all $k\ge 2$ and $\delta>0$ there is an $\eps>0$ such that the
  following holds.

  Let $0<p<1$ and let $G$ be an $n$ vertex graph.  Suppose that $G$ is
  $(\eps,p,1,2)$-pseudorandom if $k=2$, and
  $(\eps,p,k-1,2k-1)$-pseudorandom if $k>2$.  Let $U\subseteq V(G)$ be a
  vertex set of size $|U|\ge \delta n$.
  If $\vecx$ and $\vecy$ are two disjoint $k$-tuples which are
  $(\delta,p)$-connected to $U$, then there exists a $k$-path $P$ with ends
  $\vecx$ and $\vecy$ of length at most $7k$ such that $V(P)\subset
  U\cup V(\vecx)\cup V(\vecy)$.
\end{lemma}

The proof of Lemma~\ref{lem:connect} can be found in
Section~\ref{sec:connection}.  We remark that in the proof of
Theorem~\ref{thm:main} it is not especially important that the connecting
$k$-path guaranteed by this lemma is of constant length. However,
Lemma~\ref{lem:connect} is also used in the proof of
Lemma~\ref{lem:reslem}, and in this proof we need that the connecting
$k$-paths are of length independent of~$n$.

\subsection{Proof of Theorem~\ref{thm:main}}

Using Lemmas~\ref{lem:onestep},
\ref{lem:reslem}, \ref{lem:covlem} and~\ref{lem:connect} we can now prove our main theorem.

\begin{proof}
  Given $k\ge 2$ and $0<\beta< 1/2$, we set
  $\delta_{\sublem{lem:reslem}}:=\tfrac1{10}$,
  $\delta_{\sublem{lem:covlem}}:=\delta_{\sublem{lem:reslem}}^2/(10^4k)$,
  $\delta_{\sublem{lem:onestep}}:= \delta_{\sublem{lem:covlem}} /(200k)
  \le \delta^2_{\sublem{lem:reslem}} /(200k)$ and
  $\delta_{\sublem{lem:connect}}=\frac\beta{16}\cdot\delta_{\sublem{lem:reslem}}^2/(400k)$.
  We choose
  \begin{equation}\label{eq:main:eps}
    \eps\le \frac17 \cdot \frac{\beta\delta^2_{\sublem{lem:reslem}}}{6400 k^2 \cdot 2^k}
  \end{equation}
  to be small enough to apply Lemma~\ref{lem:onestep} with input $k$ and
  $\delta_{\sublem{lem:onestep}}$, to apply Lemma~\ref{lem:reslem} with
  input $k$, $\delta_{\sublem{lem:reslem}}$ and to apply Lemma~\ref{lem:covlem} $\beta$ and with input $k$,
  $\delta_{\sublem{lem:covlem}}$ and $\beta$, and to apply
  Lemma~\ref{lem:connect} with input $k$ and
  $\delta_{\sublem{lem:connect}}$.

  Let $0<p<1$ and $G$ be a graph on $n$ vertices with minimum degree at
  least $\beta p n$. If $k=2$, suppose that $G$ is
  $(\eps,p,1,2)$-pseudorandom. If $k\ge 3$, suppose that $G$ is
  $(\eps,p,k-1,2k-1)$ and $(\eps,p,k,k+1)$-pseudorandom. This ensures that
  we can apply Lemmas~\ref{lem:reslem}, \ref{lem:covlem}
  and~\ref{lem:connect}.

  Our first step now is to select an appropriate reservoir set.
  \begin{claim}
    \label{claim:reservoir}
    There is a set $R$, which we call reservoir set, such that  
    \begin{enumerate}[label=\rom]
    \item\label{itm:main:sizeR} $\delta^2_{\sublem{lem:reslem}} n/(200k)\le
      |R|\le \delta_{\sublem{lem:reslem}} n/(200k)$,
    \item\label{itm:main:degR} $\deg_R(v)\ge\frac1{2}\beta p|R|$ for all $v\in V(G)\setminus R$ and
    \item \label{itm:main:degV} $\deg_{V(G)\setminus R}(v)\ge \frac12\beta p n$ for all $v\in R$.
    \end{enumerate}
  \end{claim}
  \begin{claimproof}
    We start with an arbitrary set $R'$ of
    $2\cdot\delta^2_{\sublem{lem:reslem}} n/(200k)$ vertices.  We remove from
    $R'$ all vertices $v\in R'$ such that $\deg_{V(G)\setminus R'}(v)<3\beta
    pn/4$ to obtain $R''$. We now let $R$ be obtained from $R''$ by adding
    all vertices $v$ of $V(G)\setminus R''$ such that $\deg_{R''}(v)<\beta p
    |R''|$.
  
    We now first show that~$R$ satisfies property~\ref{itm:main:sizeR}, by
    using that~$G$ is in particular $(\eps,p,0,1)$-pseudorandom. Since
    $|V(G)\setminus R'|>3n/4>\eps n$ and $3\beta
    pn/4\le(1-\eps)p(3n/4)<(1-\eps)p|V(G)\setminus R'|$ we infer from
    Lemma~\ref{lem:pseudorandom} that 
    \begin{equation}\label{eq:main:R'R''}
      |R'\setminus R''|<\eps pn\,.
    \end{equation}
    Thus, clearly $|R'\setminus R''|< \delta^2_{\sublem{lem:reslem}}
    n/(200k)$, and hence $|R|\ge|R''|>\delta^2_{\sublem{lem:reslem}}
    n/(200k)$.  Similarly $|R''|>\eps n$ and $\beta p|R''|<(1-\eps)p| R''|$
    in conjunction with Lemma~\ref{lem:pseudorandom} implies that
    \begin{equation}\label{eq:main:RR''}
      |R\setminus R''|<\eps pn\,,
    \end{equation}
    and so $|R\setminus R''|\le \delta^2_{\sublem{lem:reslem}} n/(200k)$
    and hence $|R|=|R''|+|R\setminus R''|\le|R'|+|R\setminus R''|\le \eps p n+2\cdot
    \delta^2_{\sublem{lem:reslem}} n/(200k)\le \delta_{\sublem{lem:reslem}}
    n/(200k)$.  This yields property~\ref{itm:main:sizeR}.
  
    For~\ref{itm:main:degR} observe that $|R\setminus R''|<\eps p n$ and
    $|R|\ge \delta^2_{\sublem{lem:reslem}} n/(200k)\ge 2\eps n$ implies
    $|R''|\ge\frac12|R|$. Since $R''\subset R$ we thus have by construction
    for each $v\in V(G)\setminus R$ that $\deg_{R}(v)\ge\beta p|R''|\ge\frac12\beta
    p|R|$.

    It remains to argue that~$R$ also satisfies~\ref{itm:main:degV}.  By
    construction all vertices of $R''$ have at least $3\beta p n/4$
    neighbours in $V(G)\setminus R'$, and thus by~\eqref{eq:main:R'R''} at
    least $3\beta p n/4-\eps p n>\beta pn/2$ neighbours in $V(G)\setminus
    R$. All vertices of $R\setminus R''$, on the other hand, have at most
    $\beta p |R''|\le\frac14\beta pn$ neighbours in $R''$, and
    by~\eqref{eq:main:RR''} at most $\eps p n \le\frac14\beta pn$
    neighbours in $R\setminus R''$. Since $\delta(G)\ge\beta p n$, we
    conclude that every vertex of $R$ has at least $\frac12\beta p n$
    neighbours in $V(G)\setminus R$.
  \end{claimproof}

  We now construct a reservoir path for this reservoir~$R$ by applying
  Lemma~\ref{lem:reslem} with input $k$, $\delta_{\sublem{lem:reslem}}$,
  $\beta$, $p$, $G$ and~$R$. Observe that this is possible by
  properties~\ref{itm:main:sizeR} and~\ref{itm:main:degV} of
  Claim~\ref{claim:reservoir}.  Hence we obtain a $k$-path $P$ in $G$ which
  satisfies all four conclusions of Lemma~\ref{lem:reslem}. Let $\vecu$ be
  the start $k$-tuple of $P$, and $\vecv$ the end $k$-tuple.  We conclude
  from~\ref{reslem:connR} and~\ref{reslem:connV} of
  Lemma~\ref{lem:reslem} that $\vecu$ and $\vecv$ are
  $(\tfrac12,p)$-connected to $R$ and $(\tfrac18,p)$-connected to
  $L_1:=V(G)\setminus V(P)$.

Our next step is to extend this reservoir path to an almost spanning
$k$-path~$P'$ by repeatedly applying Lemma~\ref{lem:onestep}.  For this
purpose we let $t:=|L_1|-\delta_{\sublem{lem:covlem}} n/(200k)$ 
and apply Lemma~\ref{lem:onestep} exactly~$t$ times with $k$,
$\delta_{\sublem{lem:onestep}}$ and $p$ to $G$.  First we apply this lemma
with sets $L_1$ and $R$, and the $k$-tuple $\vecv=:(v_1,\ldots,v_k)$.  We
obtain a vertex $v_{k+1}\in L_1\cap N(v_1,\dots,v_k)$ such that
$(v_2,\dots,v_{k+1})$ is $(\frac16,p)$-connected to both~$L_1$ and~$R$.
Let $L_2:=L_1\setminus\{v_{k+1}\}$, and extend $P$ by $v_{k+1}$ to obtain
$P_1:=(P,v_{k+1})$. Similarly for each $2\le i\le t$ in succession we apply
Lemma~\ref{lem:onestep} with $L_i$, $R$ and $(v_i,\ldots,v_{k+i-1})$ and
obtain from this lemma an extending vertex $v_{k+i}$ such that
$(v_{i+1},\dots,v_{k+i})$ is $(\frac16,p)$-connected to both~$L_i$ and~$R$.
We then let $L_{i+1}:=L_i\setminus\{v_{k+i}\}$ and
$P_i:=(P_{i-1},v_{k+i})$. We need to argue that these applications of
Lemma~\ref{lem:onestep} are possible.  Indeed, by
Claim~\ref{claim:reservoir}~\ref{itm:main:sizeR} and the choice of our
constants we have $|R|\ge \delta^2_{\sublem{lem:reslem}} n/(200k) \ge
\delta_{\sublem{lem:onestep}} n$ and
$|L_i|\ge|L_1|-t=\delta_{\sublem{lem:covlem}} n/(200k)=
\delta_{\sublem{lem:onestep}} n$.  Moreover, for $i>1$ the $k$-tuple
$(v_i,\ldots,v_{k+i-1})$ is $(\tfrac16,p)$-connected to both $L_{i-1}$ and
$R$ by construction. Since $|L_i|=|L_{i-1}|-1$, the $k$-tuple
$(v_i,\ldots,v_{k+i-1})$ is thus $(\tfrac18,p)$-connected to $L_i$ (for
$i=1$ the statement is guaranteed by Lemma~\ref{lem:reslem} which constructed
$P$).

What did we achieve so far? Let $P':=P_t$ and $L:=V(G)\setminus V(P')$ be
the set of leftover vertices at this point.  Then 
\begin{equation}\label{eq:main:cov1}
  |L|\le\delta_{\sublem{lem:covlem}} n/(200k)
\end{equation}
and by Claim~\ref{claim:reservoir}\ref{itm:main:degR} every vertex of $L$
has at least $\frac12\beta p |R|$ neighbours in $R$.  By construction $P'$
is a $k$-path extending the reservoir path $P$ and covering all vertices
of~$G$ but~$L$. In addition, the start $k$-tuple $\vecu$ and end $k$-tuple
$\vecv'$ of~$P'$ are both $(\tfrac18,p)$-connected to~$R$. Clearly this
implies that these $k$-tuples are also
$\big(\tfrac{\beta}{16},p\big)$-connected to~$R$, and in the following we
will only work with this weaker conclusion.

Our next step will be to cover the leftover vertices~$L$ with a $k$-path
$P''$ using the Covering lemma, Lemma~\ref{lem:covlem}. However, this needs
some preparation. Recall that in Lemma~\ref{lem:covlem} we can choose a
vertex set~$S$ so that the $k$-path that this lemma constructs only uses
vertices from~$L$ and~$S$. As explained earlier we want to choose a big
subset of the reservoir~$R$ as~$S$. However, we need to bear in mind that
we later want to connect the start $\vecu$ of~$P'$ and the end $\vecv''$
of~$P''$ using only vertices from $U:=R\setminus V(P'')$ with the help of
the Connection lemma, Lemma~\ref{lem:connect} (similarly for the
end~$\vecv'$ of~$P'$ and the start~$\vecu''$ of~$P''$). But this lemma
requires that $\vecu$ is well-connected to~$U$. In order to guarantee this
property we will now set aside a set $R_{\vecu}\subset R$ (and similarly a
set $R_{\vecv'}$) of vertices which witness the well-connectedness
of~$\vecu$ to~$R$ and prevent these vertices from being used in~$P'$ by
setting $S=R\setminus(R_{\vecu}\cup R_{\vecv'})$.  

More precisely, recall that the
$\big(\tfrac{\beta}{16},p\big)$-connectedness of~$\vecu$ means that there
is a set of $\tfrac{\beta}{16}(p/2)^k|R|$ common neighbours of~$\vecu$
in~$R$, a set of $\tfrac{\beta}{16}(p/2)^{k-1}|R|$ common neighbours of
$(u_2,\ldots,u_k)$ in~$R$, and so on. By Remark~\ref{rem:witness} there is
a set $R_{\vecu}$ of $\frac\beta{16}(p/2)|R|$ vertices of $R$ which witness
that $\vecu$ is $\big(\tfrac{\beta}{16},p\big)$-connected to
$R$. Similarly, there is a set $R_{\vecv'}$ of $\frac\beta{16}(p/2)|R|$
vertices of $R$ which witness that $\vecv'$ is
$\big(\tfrac{\beta}{16},p\big)$-connected to $R$.  Moreover, the deletion
of any set of at most $\tfrac{\beta}{32}(p/2)^k|R|$ vertices from
$R_{\vecu}$ (or $R_{\vecv'}$) results in a set that still witnesses that
$\vecu$ is $\big(\tfrac{\beta}{32},p\big)$-connected to $R$.

Now let $S:=R\setminus(R_{\vecu}\cup R_{\vecv'})$ and note that by
part~\ref{itm:main:sizeR} of Claim~\ref{claim:reservoir} we have that
\begin{equation}\label{eq:main:cov2}
  |S|\ge|R|-\frac\beta{16}p|R|
  \ge\frac12|R|
  \ge\frac{\delta_{\sublem{lem:reslem}}^2}{400k}n
  \ge \delta_{\sublem{lem:covlem}} n\,.
\end{equation}
Moreover, since every vertex of $L$
has at least $\frac12\beta p|R|$ neighbours in $R$ we conclude from
Claim~\ref{claim:reservoir}\ref{itm:main:sizeR} that
every vertex of $L$ also has at least
\begin{equation}\label{eq:main:cov3}
  \frac\beta{2}p|R|-\frac\beta{16}p|R|
  \ge \frac{7\beta}{16}p\frac{\delta_{\sublem{lem:reslem}}^2}{200k}n
  \ge \frac12\beta\delta_{\sublem{lem:covlem}} p n
\end{equation}
neighbours in~$S$. 

 It follows from~\eqref{eq:main:cov1}, \eqref{eq:main:cov2} and
 \eqref{eq:main:cov3} that we can apply Lemma~\ref{lem:covlem} with input
 $k$, $\delta_{\sublem{lem:covlem}}$, $\beta$, $p$, $G$, $L$ and~$S$.  We
 obtain a $k$-path $P''$ with 
\begin{equation*}
\big|V(P'')\big|\le 50k|L|
\leByRef{eq:main:cov1} \frac{\delta_{\sublem{lem:covlem}}n}{4}
<\frac{1}{8}\cdot\delta_{\sublem{lem:reslem}}^2\frac{n}{200k}
\le\frac18|R|\,,
\end{equation*}
which
 covers~$L$ and whose remaining vertices are in~$S$. The start and end
 tuples $\vecu''$ and $\vecv''$ of~$P''$ are $(\tfrac18,p)$-connected and
 hence $(2\delta_{\sublem{lem:connect}},p)$-connected to
 $S\setminus V(P'')$. It follows from the choice of $R_{\vecu}$ and
 $R_{\vecv'}$ that also~$\vecu$ and~$\vecv'$ are $(\tfrac\beta{16},p)$-connected
 and hence $(\delta_{\sublem{lem:connect}},p)$-connected to  $R\setminus V(P'')$.

 Now we would like to apply Lemma~\ref{lem:connect} twice to connect the
 ends of~$P'$ and~$P''$ such that the connections use vertices from
 $R\setminus V(P'')$. For this observe that $|R\setminus
 V(P'')|\ge\frac78|R|\ge\frac78\delta^2_{\sublem{lem:reslem}} n/(200k)
 \ge2\delta_{\sublem{lem:connect}}n$. Moreover, $\vecu$ and $\vecv''$ are
 both $(2\delta_{\sublem{lem:connect}},p)$-connected to $R\setminus V(P'')$.
 Hence we can apply Lemma~\ref{lem:connect} with~$k$ and
 $\delta_{\sublem{lem:connect}}$ to find a $k$-path $C$ of length at most $7k$ connecting 
$\vecu$ and $\vecv''$ in $R\setminus V(P'')$. 
By Remark~\ref{rem:pseudo} we have $\eps^{-1}<p^k n$ and hence we can use
Claim~\ref{claim:reservoir}\ref{itm:main:sizeR} to conclude that
\begin{equation*}
  |C|\le7k
  \leByRef{eq:main:eps}\frac1\eps\cdot\frac{\beta}{32}\cdot\frac1{2^k} \delta^2_{\sublem{lem:reslem}} \frac1{200k}
  \le\frac{\beta}{32}\Big(\frac{p}{2}\Big)^k \delta^2_{\sublem{lem:reslem}} \frac{n}{200k}
  \le\frac{\beta}{32}\Big(\frac{p}{2}\Big)^k|R| \,.
\end{equation*}
It follows that $R_{\vecu}\setminus C$ and $R_{\vecv'}\setminus C$ still
witness that~$\vecu$ and~$\vecv'$ are $(\tfrac\beta{32},p)$-connected and
hence $(\delta_{\sublem{lem:connect}},p)$-connected to $R\setminus
\big(V(P'')\cup C\big)$. Hence we can apply Lemma~\ref{lem:connect} again
to find~$C'$ connecting
 $\vecu''$ and $\vecv'$ in $R\setminus \big(V(P'')\cup C\big)$.

Finally, the graph obtained by concatenating $P',C',P'',C$ certainly covers $V(G)$,
and is almost a Hamilton $k$-cycle except that some vertices in $R$ are used both in $P'$
and elsewhere. But now we can appeal to the reservoir
property~\ref{reslem:W} of the reservoir path~$P$ contained in~$P'$ to obtain a
$k$-path $P^*$ whose start and end tuples are those of $P'$, and which uses
exactly the vertices of $P'$ not in $R\cap\big(C\cup C'\cup P''\big)$. The
object obtained by concatenating $P^*,C',P'',C$ then is the desired Hamilton
$k$-cycle, and the proof is complete.
\end{proof}

\section{Proof of  Lemma~\ref{lem:connect}}\label{sec:connection}
In this section we prove the Connection lemma, Lemma~\ref{lem:connect}. We treat
the cases $k=2$ and $k\ge 3$ separately, and will first prove the case $k=2$.

\subsection{The Connection lemma for \texorpdfstring{$k=2$}{k=2}}
The idea of the proof is as follows. We want to connect two pairs
$(x_1,x_2)$ and $(y_1,y_2)=:(x'_2,x'_1)$ which are $(\delta,p)$-connected
to a large set~$U$ of vertices, i.e.\ $|N(x_2)\cap U|,|N(x_2')\cap
U|\ge\delta \tfrac{p}{2}|U|$, and $|N(x_1,x_2)\cap U|,|N(x_1',x_2')\cap
U|\ge\delta \big(\tfrac{p}{2}\big)^2|U|$.  For this we identify disjoint
sets $X_3,\ldots,X_7$ in $U$ and create many $2$-paths
$(x_1,x_2,\dots,x_7)$ with $x_i\in X_i$ for $3\le i\le 7$ as follows.  We
let $X_3$ consist of $\Omega(p^2n)$ vertices in $N(x_1,x_2)$, $X_4$ of
$\Omega(pn)$ vertices in $N(x_2)$, and~$X_5$, $X_6$ and~$X_7$ of
$\Omega(n)$ vertices. Now any vertex $x_3\in X_3$ has the property that
$(x_1,x_2,x_3)$ is a $2$-path, and most of these vertices have about the
expected number of neighbours in~$X_4$. Our pseudorandomness condition then
implies that we can find $\Omega(pn)$ vertices $x_4\in X_4$ such that
$(x_1,x_2,x_3,x_4)$ is a $2$-path.  Similarly, $\Omega(n)$ vertices of
$X_5$ are the end vertex of a $2$-path from $(x_1,x_2)$ through $X_3$ and
$X_4$, and extending these paths further to~$X_6$ we obtain that most
vertices of $X_6$ are ends of $2$-paths from $(x_1,x_2)$.

Analogously we construct sets $X'_3,\dots,X'_7$ and $2$-paths through these
sets extending $(x_1',x_2')$. It remains to connect one of the $2$-paths
extending $(x_1,x_2)$ and one extending $(x'_1,x'_2)$. It seems plausible
that this should be possible because we have so many candidates for these
$2$-paths. However, so far we only know that most \emph{vertices} in $X_7$
are ends of $2$-paths from $(x_1,x_2)$. But in order to connect two $2$-paths
information merely about the final vertex of each of the paths is not
enough, but we need information about the last \emph{edge} of the paths. 
To this end we actually prove the  following stronger property
for~$X_6$. We can find a subset $Y_6$ of $\Omega(pn)$ vertices $x_6$ in
$X_6$ with the following property. There are $\Omega(p^2n)$ vertices $x_5$ of
$X_5$ such that $(x_5,x_6)$ is the end of a $2$-path from $(x_1,x_2)$ -- we call such edges
$x_5x_6$ \emph{good}.  Similarly we find $Y'_6\subset X'_6$.

This stronger property then enables us to show that \emph{almost all edges}
from $Y_6$ to $X_7$ are ends of $2$-paths from $(x_1,x_2)$ and almost all
edges from $Y'_6$ to $X'_7$ are ends of $2$-paths from $(x'_1,x'_2)$.
Since $Y_6$ and $Y'_6$ are still only of size $\Omega(pn)$, we repeat this
argument and obtain similar sets $Y_7\subset X_7$ and $Y'_7\subset X'_7$ of
size $\Omega(n)$, such that most edges from $Y_7$ to $X'_7$ are ends of
$2$-paths from $(x_1,x_2)$ and most edges from $Y'_7$ to $X_7$ are ends of
$2$-paths from $(x_1',x_2')$.  Since $Y_7$ and $Y'_7$ are both large, we
can then use the pigeonhole principle to find an edge between $Y_7$ and
$Y'_7$ which is the end of a $2$-path both from $(x_1,x_2)$ and (in the
reverse direction) from $(x'_1,x'_2)$, and hence we find the desired $2$-path
connecting $(x_1,x_2)$ and $(x_2',x_1')$.

\begin{proof}[Proof of Lemma~\ref{lem:connect} for $k=2$]
  Given $\delta>0$, we set $\eps=\delta^2/10^6$.  Assume that~$G$ is
  $(\eps,p,1,2)$-pseudorandom and $|U|\ge\delta n$.  By
  Remark~\ref{rem:pseudo} this implies $p^2|U|\ge10^6\delta^{-1}$.  Let
  $\vecx$ and $\vecy$ be $(\delta,p)$-connected to~$U$.  Our goal is to find
  a connection between $\vecx=(x_1,x_2)$ and $\vecy=:(x'_2,x'_1)$.

  We first identify ten disjoint sets in $U$ in which we will find our ten
  connecting vertices. We first choose $X_3\subset \big(N(x_1,x_2)\cap
  U\big)\setminus\{x_1,x_2,x_1',x_2'\}$ and $X'_3\subset
  \big(N(x_1',x_2')\cap U\big)\setminus\big(\{x_1,x_2,x_1',x_2'\}\cup
  X_3\big)$, then $X_4\subset \big(N(x_2)\cap U\big)\setminus
  \big(\{x_1,x_2,x_1',x_2'\}\cup X_3\cup X'_3\big)$ and $X'_4\subset
  \big(N(x_2')\cap U\big)\setminus \big(\{x_1,x_2,x_1',x_2'\}\cup X_3\cup
  X'_3\cup X_4\big)$, and then pairwise disjoint subsets
  $X_5,X_6,X_7,X'_5,X'_6,X'_7$ of $U\setminus\big(\{x_1,x_2,x_1',x_2'\}\cup
  X_3\cup X'_3\cup X_4\cup X'_4\big)$, such that
  \begin{align*}
    |X_3|,|X'_3| &= \tfrac1{16}\delta p^2|U| \,, \qquad\qquad\\
    |X_4|,|X'_4| &= \tfrac1{16}\delta p|U| \,, \\
    |X_5|,|X'_5|, |X_6|,|X'_6|, |X_7|,|X'_7| &= \tfrac1{10}|U| \,.
  \end{align*}
  Here, the choice of $|X_3|$ (and similarly $|X'_3|$) is possible because
  $(x_1,x_2)$ is $(\delta,p)$-connected to~$U$ and so $|N(x_1,x_2)\cap
  U|\ge \delta p^2|U|/4$. The choice of $|X_4|$ (and
  similarly $|X'_4|$) is possible because $|N(x_1)\cap U|\ge\delta p|U|/2$
  and $X_3$, $X'_3$ are small. The choice of the remaining sets is possible
  because all previously chosen sets are small. Note that since $p^2|U|$ is
  large, so all of these sets are large and rounding errors do not affect the
  validity of this argument.
  
  By construction all vertices of~$X_3$ form a $2$-path with
  $(x_1,x_2)$. We shall now extend these $2$-paths to $X_4$, $X_5$, and so
  on.  For this let 
  \begin{equation*}
    Y_3:=\big\{y\in X_3\colon
      \deg_{X_4}(y)\ge\delta p^2|U|/20 \,,\,\,
      \deg_{X_5}(y)\ge p|U|/20
    \big\}
  \end{equation*}
  That is, the vertices in $y_3\in Y_3$ have many
  $2$-path extensions $(x_1,x_2,y_3,x_4)$ into~$X_4$ and they are good
  candidates for having many $2$-paths which extend even further to~$X_5$.
  Since $|X_4|\ge\delta p|U|/16>\eps p n$ and $\delta
  p^2|U|/20<(1-\eps)p|X_4|$ we can use Lemma~\ref{lem:pseudorandom} to
  infer that at most $\eps p^2 n$ vertices of~$X_3$ fail the first of these
  two conditions because~$G$ is $(\eps,p,1,2)$-pseudorandom. Similarly, at
  most $\eps p^2 n$ vertices fail the second condition, and hence
  \begin{equation*} 
    |Y_3|\ge |X_3|-2\eps p^2n
    \ge \tfrac1{16}\delta p^2 |U|-2\eps p^2 \delta^{-1}|U|
    \ge \delta p^2|U|/20\,.
  \end{equation*}
  Next, for each $y_3\in Y_3$, we let 
  \begin{equation*}
    Y_4(y_3) :=\big\{y\in N_{X_4}(y_3) \colon
      \deg_{X_5}(y,y_3)\ge p^2|U|/40 \,,\,\,
      \deg_{X_6}(y)\ge p|U|/20
    \big\} \,.
  \end{equation*}
  Observe that for each vertex $y_4\in Y_4(y_3)$ we have that
  $(x_1,x_2,y_3,y_4)$ is a $2$-path and this $2$-path is a good candidate
  for having many extensions to $X_5$ and $X_6$. Again, since $|X_5\cap
  N(y_3)|\ge \delta p|U|/20>\eps p n$ by the definition of~$Y_3$ and
  $p^2|U|/40<(1-\eps)p|X_5\cap N(y_3)|$ we can use
  Lemma~\ref{lem:pseudorandom} to infer that at most $\eps p^2n$ vertices
  fail the first condition, and similarly for the second condition. So
  \begin{equation*}
    |Y_4(y_3)|\ge |X_4\cap N(y_3)|-2\eps p^2n
    \ge\delta p^2|U|/20-2\eps p^2n\ge \delta p^2|U|/40\,.
  \end{equation*}
  Analogously, for each $y_3\in Y_3$ and $y_4\in Y_4(y_3)$, we let
  \begin{equation*}
    Y_5(y_3,y_4) :=\big\{y\in N_{X_5}(y_3,y_4) \colon
      \deg_{X_6}(y,y_4)\ge p^2|U|/40 \,,\,
      \deg_{X_7}(y)\ge p|U|/20
    \big\} \,.
  \end{equation*}
  Similarly as before we have for each $y_5\in Y_5(y_3,y_4)$ that
  $(x_1,x_2,y_3,y_4,y_5)$ is a $2$-path and Lemma~\ref{lem:pseudorandom}
  implies $|Y_5(y_3,y_4)|\ge p^2|U|/40-2\eps p^2n\ge p^2|U|/80$.

  For $y_3\in Y_3$ we let
  $Y_5(y_3):=\bigcup_{y_4\in Y_4(y_3)}Y_5(y_3,y_4)$, and set
  $Y_5:=\bigcup_{y_3\in Y_3}Y_5(y_3)$
  and claim that 
  \begin{equation}\label{eq:con:Y5}
    |Y_5(y_3)|\ge p|U|/160 \qquad\text{and}\qquad |Y_5|\ge |U|/200\,.
  \end{equation}
  Indeed, for the first part let $y_3\in Y_3$ be fixed, assume otherwise
  and consider the set $(N(y_3)\cap X_5)\setminus Y_5(y_3)$, which has
  cardinality at least $p|U|/20-p|U|/160=7p|U|/160$. Since $|N(y_3)\cap
  X_4|\ge \delta p^2|U|/20$ by definition, we can thus use
  Lemma~\ref{lem:pseudorandom} to pick a vertex $y_4\in X_4\cap N(y_3)$
  which is ``typical'' with respect to $N(y_3)\cap X_5$ and with respect
  to~$X_6$, that is, which satisfies $|N_{X_5}(y_3,y_4)\setminus
  Y_5(y_3)|\ge p^2|U|/40$ and $\deg_{X_6}(y_4)\ge p|U|/20$. Hence, in
  particular, $y_4\in Y_4(y_3)$.  We now show that
  $N_{X_5}(y_3,y_4)\setminus Y_5(y_3)$, since it is big, contains a vertex
  from $Y_5(y_3,y_4)\subset Y_5(y_3)$, a contradiction. For this we need to
  show that there is $y_5\in N_{X_5}(y_3,y_4)\setminus Y_5(y_3)$ with
  $\deg_{X_6\cap N(y_4)}(y_5)\ge p^2|U|/40$ and $\deg_{X_7}(y_5)\ge
  p|U|/20$. But such a vertex exists by Lemma~\ref{lem:pseudorandom}
  because $|X_6\cap N(y_4)|\ge p|U|/20$ by the definition of $y_4$.
  For the second part note that each $y_3\in Y_3$ has at
  least $|Y_5(y_3)|\ge p|U|/160$ neighbours in $Y_5$, and thus we have
  $e(Y_3,Y_5)\ge |Y_3|p|U|/160$. By~\eqref{eq:pseudodisc}, we have
  $e(Y_3,Y_5)\le(1+\eps)p|Y_3|\cdot|Y_5|$, and thus $|Y_5|\ge
  |U|/200$. Hence we have~\eqref{eq:con:Y5}.
  
  We next define good edges between $X_5$ and $X_6$. Let $y_5\in Y_5$.  For
  a neighbour $x_6\in X_6$ of $y_5$, we call the edge $y_5x_6$ \emph{good}
  if $x_1x_2x_3x_4y_5x_6$ is a $2$-path for some $x_3\in X_3$ and~$x_4\in
  X_4$. For each $y_5$ in $Y_5$ there are $y_3\in Y_3$ and $y_4\in Y_4$
  such that $y_5\in Y_5(y_3,y_4)$, which means $|N(y_4,y_5)\cap X_6|\ge
  p^2|U|/40$ by definition.  So each vertex in $Y_5$ sends at least
  $p^2|U|/40$ good edges to $X_6$.  Hence the average number of good edges
  incident to a vertex of~$X_6$ is at least $|Y_5|p^2|U|/(40|X_6|)\ge
  p^2|U|/800$, where we used $|X_6|=|U|/10$ and~\eqref{eq:con:Y5}.  Let
  $Z_6$ be the set of those vertices in $X_6$ which are incident to at
  least $p^2|U|/1000$ good edges from $Y_5$. We will show that
  \begin{equation}\label{eq:con:Z6}
    |Z_6|\ge p|U|/300
  \end{equation}
  by using a double counting argument. Indeed, the total number
  $\Good(Y_5,X_6)$ of good edges from $Y_5$ to $X_6$ is at least
  $|Y_5|p^2|U|/40$.  
  By definition of~$Z_6$ each vertex in
  $X_6\setminus Z_6$ is incident to less than
  $p^2|U|/1000$ good edges.  Thus
  \begin{equation*}
    |Y_5|p^2|U|/40
    \le\Good(Y_5,X_6)
    \le e(Y_5,Z_6)
    +|X_6|p^2|U|/1000\,.
  \end{equation*}
  Now~\eqref{eq:con:Y5} implies that the second summand can be bounded by
  $|X_6|p^2|U|/1000=\frac1{200}|U|\cdot\frac1{50}p^2|U|\le |Y_5|p^2|U|/50$.
  Hence $|Y_5|p^2|U|/200\le e(Y_5,Z_6)\le |Y_5||Z_6|$ implying $|Z_6|\ge p^2|U|/200>\eps p^2 n$. 
  This allows us to immediately obtain the desired bound~\eqref{eq:con:Z6} since we can now estimate 
  $e(Y_5,Z_6)\le (1+\eps)p|Y_5||Z_6|$ using~\eqref{eq:pseudodisc}, improving thus the lower bound
   on $|Z_6|$ by a factor of $p/(1+\eps)$.

  We now let $Y_6\subset Z_6$ be the set of
  those vertices with at least $p|U|/20$ neighbours in $X_7$.
  That is, $Y_6$ is the set of those vertices in $X_6$ which receive many
  good edges from $Y_5$ and have many neighbours in $X_7$. These are the
  vertices that we will continue to work with in the following.
  Lemma~\ref{lem:pseudorandom} gives a lower bound
  \begin{equation}\label{eq:con:Y6}
    |Y_6|\ge|Z_6|-\eps p^2 n
    \geByRef{eq:con:Z6} p|U|/300-\eps p^2 n\ge p|U|/400\,,
  \end{equation}
  for the number of vertices in this set. However, this lower bound is only
  of order $\Omega(pn)$. Hence we iterate and define good edges between
  $Y_6$ and $X_7$ to obtain a linear sized set $Y_7$ with similar
  properties.
  
  Given an edge $y_6x_7$ from $Y_6$ to $X_7$, we call $y_6x_7$ \emph{good}
  if there is $y_5\in Y_5$ such that $y_5y_6$ is a good edge and $y_5$ is
  adjacent to $x_7$. By definition of $Y_6$, for $y_6\in Y_6$ there are at
  least $p^2|U|/1000>\eps p^2 n$ vertices of $Y_5$ which send good edges to
  $y_6$. It follows by~\eqref{eq:pseudodisc} that at most $\eps p n$ edges
  from $y_6$ to $X_7$ are not good, for each $y_6\in Y_6$.  Since vertices
  in~$Y_6$ have at least $p|U|/20$ neighbours in~$X_7$ we thus conclude
  that there are at least $|Y_6|(p|U|/20-\eps pn)\ge |Y_6|p|U|/40$ good
  edges from $Y_6$ to $X_7$.
  Let $Y_7\subset X_7$ be the set of those vertices which are incident to
  at least $p^2|U|/5000$ good edges (again, a bit less than the average,
  which is at least $|Y_6|p|U|/(40|X_7|)\ge p|U|/4000$). Applying a
  similar double counting argument as before, using~\eqref{eq:pseudodisc}
  and~\eqref{eq:con:Y6}, we obtain
  \begin{equation}\label{eq:con:Y7}
    |Y_7|\ge |U|/100\,.
  \end{equation}
  
  Let us examine the good edges leaving $Y_7$: We call an edge from
  $y_7\in Y_7$ to $x'_7\in X'_7$ \emph{good} if there is $y_6\in Y_6$ such
  that $y_6y_7$ is a good edge and $y_6$ is adjacent to $x'_7$. By
  definition of $Y_7$, for each $y_7\in Y_7$ there are at least
  $p^2|U|/5000$ good edges from $Y_6$ to $y_7$, and thus
  by~\eqref{eq:pseudodisc} there are at most $\eps p n$ edges from $y_7$ to
  $X'_7$ which are not good. Observe that by definition any good edge
  $y_7x'_7$ from $Y_7$ to $X'_7$ is the last edge in a $2$-path from
  $x_1x_2$ to $y_7x'_7$ using one vertex of each set $X_3,\ldots,X_6$.
  
  Now we repeat the identical construction within the sets $X'_3,\ldots,X'_7$,
  obtaining a set $Y'_7\subset X'_7$ of size at least $|U|/100$, where each
  vertex $y'_7\in Y'_7$ sends at most $\eps p n$ edges to $X_7$ which are
  not good, and each
  good edge from $y'_7$ to $X_7$ is the last edge in a $2$-path from $x_1'x_2'$
  using one vertex of each set $X'_3,\ldots,X'_6$. 

  Finally we can apply the pigeon hole principle:
  By~\eqref{eq:pseudodisc} there are at least
  \[(1-\eps)p|Y_7||Y'_7|\geByRef{eq:con:Y7}
  p|U|^2/20000>\eps p n^2>\big(|Y_7|+|Y'_7|\big)\eps p n\] edges between
  $Y_7$ and $Y'_7$, and in particular there is one edge $y_7y'_7$ which is
  both good from $Y_7$ to $Y'_7$ and good from $Y'_7$ to $Y_7$. This yields a
  $2$-path from $x_1x_2$ to $x_2'x_1'$ using one vertex of each set
  $X_3,\ldots,X_6$, $y_7$, $y'_7$, and one vertex of each set
  $X'_6,\ldots,X'_3$, as desired.
\end{proof}

\subsection{The Connection lemma for  \texorpdfstring{$k>2$}{k>2}}
We use the same general strategy as in the $k=2$ case. To connect the
$k$-tuples $\vecx$ and $\vecy$ we start by
constructing short $k$-paths from $\vecx$ step by step. In each step we
look for many possible extensions of each of the $k$-paths constructed so
far (so in step~$i$ all our $k$-paths will be of length~$i$). Our goal is
to continue until we reach a collection of
$k$ disjoint $\Omega(n)$-sized vertex subsets of $U$ such that\\
($\star$) most copies
of $K_k$ with one vertex in each of the $k$ sets are ends of $k$-paths leaving
$\vecx$.\\
Repeating from $\vecy$, the pigeonhole
argument then guarantees that one of these copies of $K_k$ is also the end of a
$k$-path leaving $\vecy$ in the reverse order, and thus we get the desired
$\vecx$-$\vecy$ connection. 

However, obtaining property ($\star$) is not straightforward. 
In fact $(\eps,p,k-1,k)$-pseudorandomness,  a weaker pseudorandomness condition than we require, would be enough
to guarantee that after $k+1$ steps we get $k$-paths from $\vecx$
to a set of $\Omega(n)$ vertices. 
Thus after  $k-1$ further steps we get $k$ disjoint $\Omega(n)$-sized
subsets of~$U$ of vertices which are the ends of $k$-paths from~$\vecx$,
and we might hope that these sets also satisfy property ($\star$).
However we are not
able to show this with this weaker pseudorandomness condition.

Hence we resort to demanding $(\eps,p,k-1,2k-1)$-pseudorandomness. This
allows us to show an inductive version of ($\star$): at each step we maintain the
property that most copies of $K_k$ in the final $k$ sets are ends of
$k$-paths from $\vecx$.

\smallskip

The inductive argument as well as the pigeonhole argument in
this proof rely on the following proposition, which states that in a
sufficiently pseudorandom graph every collection of~$k$ sufficiently large
disjoint vertex sets span roughly the expected number of $k$-cliques.
We use the following definitions. For a graph $G$ and disjoint
subsets $V_1,\ldots,V_k$ of the vertex set $V(G)$ we denote by
$K_k(V_1,\ldots,V_k)$ the set of all copies of $K_k$ crossing
$V_1,\ldots,V_k$, i.e., with one vertex in each of the sets
$V_1,\ldots,V_k$. Given $p\in[0,1]$, we define
\[\EK_k(V_1,\ldots,V_k):=p^{\binom{k}{2}}\prod_{i=1}^k|V_i|\,,\]
which we call the \emph{expected number of $k$-cliques crossing
  $V_1,\ldots,V_k$}.

\begin{proposition}\label{prop:cliquecount}
  For each $\mu>0$ and integer $k\ge 1$ there exists $\eps>0$ such that for
  all $p\in(0,1)$ the following holds. Suppose that $k\ge r\ge 2$ is an
  integer, and that $V_1,\ldots,V_r$ are pairwise disjoint vertex sets in
  an $(\eps,p,k-1,2k-2)$-pseudorandom graph $G$ on $n$ vertices such that
  $|V_i|\ge \mu p^{k-i}n$ for each $r\ge i\ge 1$. Then we have
\[\big|K_r(V_1,\ldots,V_r)\big|=(1\pm\mu)\EK_r(V_1,\ldots,V_r)\,.\]
\end{proposition}

We remark that the lower bound in this proposition requires only
$(\eps,p,k-2,k-1)$-pseudorandomness and that also the pseudorandomness
requirement for the upper bound can undoubtedly be improved.

\begin{proof}[Proof of Proposition~\ref{prop:cliquecount}]
  Given $0<\mu\le 1$, we take $0<\eps_0<2^{-k}\mu$ small enough so that
  $(1\pm2k\eps_0/\mu)^{\binom{k+1}{2}}$ is a sub-range of $1\pm\mu$. Given
  $0<\eps<\eps_0$, we will prove by induction on $r$ the stronger statement
  \[\big|K_r(V_1,\ldots,V_r)\big|=(1\pm\tfrac{2k\eps}\mu)^{\binom{r+1}{2}}\EK_r(V_1,\ldots,V_r)\]
  for disjoint sets $V_1,\ldots,V_r$ in an $(\eps,p,k-1,2k-2)$-pseudorandom graph
  $G$ with $|V_i|\ge2^{r-k}\mu p^{k-i}n$ for each $i$. The base case $r=2$ is
  immediate from $(\eps,p,k-2,k-1)$-pseudorandomness.
  
  For the induction step, we split the vertices of $V_1$ into two classes: the
  typical vertices, whose degree into $V_i$ is $(1\pm\eps)p|V_i|$ for each $2\le
  i\le r$, and the remaining atypical vertices. Since $(1-\eps)p|V_i|\ge
  p|V_i|/2\ge 2^{r-1-k}\mu p^{k-i+1}n$, for each typical vertex $v$ we have by
  induction the estimate 
  \begin{equation*}\begin{split}
    \big|K_{r-1}\big(N_{V_2}(v),\ldots,N_{V_r}(v)\big)\big|
    & =(1\pm2k\eps/\mu)^{\binom{r}{2}}
          \big|\EK_r\big(N_{V_2}(v),\ldots,N_{V_r}(v)\big)\big| \\
    & =(1\pm2k\eps/\mu)^{\binom{r}{2}}p^{\binom{r-1}{2}}|N_{V_2}(v)|\dots|N_{V_r}(v)| \\
    & =(1\pm2k\eps/\mu)^{\binom{r}{2}}p^{\binom{r}{2}}(1\pm\eps)^{r-1}|V_2|\cdots|V_r|\,,
  \end{split}\end{equation*}
  which is the contribution of $v$ to $\big|K_r(V_1,\ldots,V_r)\big|$. 
  By Lemma~\ref{lem:pseudorandom} 
  all but at most $(r-1)\eps p^{2k-2}n\le 2k\eps p^{k-1}\mu^{-1}|V_1|$
  vertices of $V_1$ are typical. This clearly already yields the lower
  bound of our proposition. 

  To obtain the upper bound, it is then enough to show that the atypical
  vertices do not contribute too much. An atypical vertex certainly does
  not contribute more than $\big|K_{r-1}(V_2,\ldots,V_r)\big|$, which by
  induction is not more than
  \[(1+\tfrac{2k\eps}\mu)^{\binom{r}{2}}p^{\binom{r-1}{2}}|V_2|\cdots|V_r|\,.\] 
  Hence we get
  \begin{equation*}\begin{split}
    \big|K_r(V_1,\ldots,V_r)\big|
    &\le \big(1+\tfrac{2k\eps}{\mu}\big)^{\binom{r}{2}}p^{\binom{r}{2}}(1+\eps)^{r-1}|V_1|\cdots|V_r|\\
    &\qquad\quad +\tfrac{2k\eps}{\mu}p^{k-1}|V_1|\big(1+\tfrac{2k\eps}{\mu}\big)^{\binom{r}{2}}
                     p^{\binom{r-1}{2}}|V_2|\cdots|V_r|\\
    &\le \Big(\big(1+\tfrac{2k\eps}{\mu}\big)^{\binom{r}{2}+r-1}
                   + \tfrac{2k\eps}{\mu}\big(1+\tfrac{2k\eps}{\mu}\big)^{\binom{r}{2}}\Big)
            p^{\binom{r}{2}}|V_1|\cdots|V_r|\\
    &\le\big(1+\tfrac{2k\eps}{\mu}\big)^{\binom{r+1}{2}}p^{\binom{r}{2}}|V_1|\cdots|V_r|
  \end{split}\end{equation*}
  as desired.
\end{proof}

We now give the proof of the Connection lemma in the case $k>2$, modulo
a claim which encapsulates the inductive argument, whose proof we will
provide subsequently.

\begin{proof}[Proof of Lemma~\ref{lem:connect} for $k>2$]
Let $k>2$ and $0<\delta\le 1/(6k)$ be given. We set
$\xi_{k+1}:=\tfrac{1}{3}$, and for each $k+1\ge i\ge 2$, we set
\begin{equation}\label{eq:setxiconn}
\xi_{i-1}:=\tfrac{1}{4}\xi_i^{k-1}3^{-\binom{k}{2}}\,.
\end{equation}
We choose
\begin{equation}\label{eq:setmuepsconn}
  \mu:=\big( \tfrac{1}{10k}10^{-10k^2}\delta^2\xi_1 \big)^2 \quad
  \text{and}\quad\eps\le\mu
\end{equation}
to be small enough for Proposition~\ref{prop:cliquecount} with input $\mu$ and $k$.
Let $0<p<1$ and $G$ be an $(\eps,p,k-1,2k-1)$-pseudorandom graph on $n$
vertices. Let $U$ be a subset of $V(G)$ of size $|U|\ge\delta n$. Suppose that
$\tpl{x}$ and $\tpl{y}$ are disjoint $k$-tuples which are $(\delta,p)$-connected to $U$.

We choose pairwise disjoint subsets $U_1,\ldots,U_{2k},U'_1,\ldots,U'_k$ of $U$
with
\begin{alignat}{2}
  |U_i|,|U'_i|&=\tfrac{1}{3k}\delta^2(p/2)^{k-i+1}n
  & \qquad &\text{for $i\le k$, and} \label{eq:con:Ui:small} \\
  |U_i|&=\tfrac{1}{3k}\delta n
  &&\text{for $i>k$}  \label{eq:con:Ui:large}
\end{alignat}
as follows. We first choose the disjoint sets $U_1$ in $U\cap
N(x_1,\ldots,x_k)$ and $U'_1$ in $U\cap N(y_1,\ldots,y_k)$. From the
remaining vertices in $U$ we then choose the disjoint sets $U_2$ in $U\cap
N(x_2,\ldots,x_k)$ and $U'_2$ in $U\cap N(y_2,\ldots,y_k)$. We continue in
this fashion, choosing for each $i\le k$ the set $U_i$ in $U\cap
N(x_i,\ldots,x_k)$ and the set $U'_i$ in $U\cap
N(y_i,\ldots,y_k)$. Choosing these sets is possible by the
$(\delta,p)$-connectedness of $\tpl{x}$ and $\tpl{y}$ to $U$. Finally we
choose in the remaining vertices of~$U$ disjoint sets
$U_{k+1},\ldots,U_{2k}$ arbitrarily of the prescribed size.  Further, for
each $i\in[k]$ we let $U'_{k+i}:=U_{2k-i+1}$. To summarise, we constructed
$3k$ disjoint sets which we will use to construct $k$-paths: we will find
many $k$-paths starting in $\tpl x$ with one vertex in each of
$U_1,\dots,U_{2k}$ (that is why we chose $U_1,\dots,U_k$ in the
neighbourhood of vertices from~$\tpl x$), and many $k$-paths starting in
$\tpl y$ using $U'_1,\dots,U'_{2k}$. We will argue that, since
$U_{k+1},\dots,U_{2k}$ and $U'_{2k},\dots,U'_{k+1}$ coincide, two of these
$k$-paths join.

More precisely, for each $1\le i\le k+1$, we call a $k$-clique $\tpl{c}$ in
$K_k(U_i,\ldots,U_{i+k-1})$ \emph{good} (with respect to $\vecx$) if there
is a $k$-path from $\vecx$ with one vertex in each of $U_1,\ldots,U_{i-1}$
followed by $\tpl{c}$, in that order, and \emph{bad} otherwise.  We will
use the following claim, whose proof we postpone.

\begin{claim}\label{clm:conn:indhyp} For each $1\le
i\le k+1$, all but at most
$\xi_i \EK_k(U_i,\ldots,U_{i+k-1})$
of the $k$-cliques in $K_k(U_i,\ldots,U_{i+k-1})$ are good.
\end{claim}

This claim implies the desired statement.  Indeed, by
Claim~\ref{clm:conn:indhyp} all but at most
$\tfrac{1}{3}\EK(U_{k+1},\ldots,U_{2k})$ of the $k$-cliques in
$K_k(U_{k+1},\ldots,U_{2k})$ are good with respect to $\vecx$. Similarly, 
for each $1\le i\le k$ we call a clique in $K_k(U'_i,\ldots,U'_{i+k-1})$ \emph{good
with respect to $\vecy$} if it is the end of a $k$-path from $\tpl{y}$
using one vertex in each of $U'_1,\ldots,U'_{i+k-1}$ in that order. By symmetry
Claim~\ref{clm:conn:indhyp} guarantees that also all but at most
$\tfrac{1}{3}\EK(U_{k+1},\ldots,U_{2k})$ of the $k$-cliques in
$K_k(U'_{k+1},\ldots,U'_{2k})=K_k(U_{2k},\ldots,U_{k+1})$ are good with
respect to $\vecy$. By Proposition~\ref{prop:cliquecount}, there are at
least
\[(1-\mu)\EK(U_{k+1},\ldots,U_{2k})>\tfrac{2}{3}\EK(U_{k+1},\ldots,U_{2k})\]
cliques in $K_k(U_{k+1},\ldots,U_{2k})$, and therefore there must exist a
clique which is both good with respect to $\vecx$ and to $\vecy$. Hence we
obtain the desired $(\vecx-\vecy)$-connecting $k$-path.
\end{proof}

It remains to establish Claim~\ref{clm:conn:indhyp}, which we prove
by induction on~$i$.  

\begin{proof}[Proof of Claim~\ref{clm:conn:indhyp}]
  For the base case $i=1$, observe that by definition of the sets $U_1,\ldots,U_k$ there are no
  bad cliques in $K_k(U_1,\ldots,U_k)$.

  For the induction step, assume $2\le i\le k+1$. Let
  $W_0:=U_{i-1},\ldots,W_k:=U_{i+k-1}$. Suppose for contradiction that
  $K_k(W_1,\ldots,W_k)$ contains at least $\xi_i \EK_k(W_1,\ldots,W_k)$ bad
  cliques. We shall show that this implies at least $\xi_{i-1}
  \EK_k(W_0,\ldots,W_{k-1})$ bad cliques in $K_k(W_0,\ldots,W_{k-1})$,
  contradicting the induction hypothesis.

  To this end we shall find many cliques~$\tpl{c}$  of size $k-1$ in
  $K_{k-1}(W_1,\ldots,W_{k-1})$ with the following two properties.
  Firstly, $\tpl{c}$ has a set $C'_0(\tpl{c})$ of common neighbours in
  $W_0$ of size at least $(1-\eps)^{k-1}p^{k-1}|W_0|$ (i.e.\ almost the
  expected number). Secondly, there is a set $C_k(\tpl{c})$ of vertices in
  $W_k$ with $|C_k(\tpl{c})|\ge\eps p^{k-1}n$ (i.e.\ a small but constant fraction
  of the average) such that $(\tpl{c},c_k)$ forms a bad clique for each
  $c_k\in C_k(\tpl{c})$. If a $(k-1)$-clique $\tpl{c}$ has these two properties we also say
  that $\tpl{c}$ is a \emph{normal clique}.
  \begin{claim}\label{cl:normal}
    $K_{k-1}(W_1,\ldots,W_{k-1})$ contains at least 
    \[\prod_{j=1}^{k-1}\xi_i3^{-j}p^{j-1}|W_j|=\xi_i^{k-1}3^{-\binom{k}{2}}p^{\binom
      {k-1}{2}}|W_1|\cdots|W_{k-1}| \]
    normal $(k-1)$-cliques.
  \end{claim}
  Before proving this claim we argue that this implies the desired
  contradiction. Indeed, let $\tpl{c}$ be a normal $(k-1)$-clique in
  $K_{k-1}(W_1,\ldots,W_{k-1})$. Then by definition we have 
  \begin{equation*}\begin{split}
    |C'_0(\tpl{c})| &\ge(1-\eps)^{k-1}p^{k-1}|W_0|
    \ge (1-\eps)^{k-1}p^{k-1}\big(\tfrac12|W_0|+\tfrac12|U_1|\big) \\
    & \geByRef{eq:con:Ui:small}(1-\eps)^{k-1}p^{k-1}
        \left(\frac12|W_0|+
        \frac12\cdot\frac{\delta^2}{3k\cdot 2^k} p^{k} n \right) \\
    & \geByRef{eq:setmuepsconn}
    \tfrac{1}{2}(1-\eps)^{k-1}p^{k-1}|W_0|+\eps p^{2k-1}n
  \end{split}\end{equation*}
  and 
  \begin{equation*}
    |C_k(\tpl{c})|
    \ge\xi_i3^{1-k}p^{k-1}|W_k|
    \geByRef{eq:con:Ui:large} \xi_i 3^{-k} \cdot \frac{1}{k}\delta \cdot p^{k-1} n
    \geByRef{eq:setmuepsconn} \eps p^{k-1}n \,.
  \end{equation*}
  Thus, since~$G$ is $(\eps,p,k-1,2k-1)$-pseudorandom,
  Lemma~\ref{lem:pseudorandom} implies that at most $\eps p^{2k-1}n$
  vertices of $C'_0(\tpl{c})$ do not have any neighbours in
  $C_k(\tpl{c})$. It follows that the set $C_0(\tpl{c})$ of vertices in
  $C'_0(\tpl{c})$ which do have neighbours in $C_k(\tpl{c})$ has size at
  least $\tfrac{1}{2}(1-\eps)^{k-1}p^{k-1}|W_0|$.

  Why are we interested in these edges $c_0c_k$ between $C_0(\tpl{c})$ and
  $C_k(\tpl{c})$? By definition of $C_k(\tpl{c})$ the $k$-clique
  $(\tpl{c},c_k)$ is a bad $k$-clique in $K_k(W_1,\ldots,W_k)$. Hence,
  since by definition of $C'_0(\tpl{c}) \supseteq C_0(\tpl{c})$ we have
  $c_0\in N_{W_0}(\tpl{c})$, the edge $c_0c_k$ witnesses that also the
  $k$-clique $(c_0,\tpl{c})$ must be bad (in
  $K_{k-1}(W_1,\ldots,W_{k-1})$).
  Because
  $\EK_k(W_0,\ldots,W_{k-1})=p^{\binom{k}{2}}\prod_{i=0}^{k-1}|W_i|$ by
  definition, it therefore follows from Claim~\ref{cl:normal} that we find at
  least
  \begin{equation*}\begin{split}
    |C'_0(\tpl{c})| \cdot &\xi_i^{k-1}3^{-\binom{k}{2}}p^{\binom{k-1}{2}}|W_1|\cdots|W_{k-1}| \\
    &\ge \tfrac{1}{2}(1-\eps)^{k-1}p^{k-1}|W_0|\cdot
         \xi_i^{k-1}3^{-\binom{k}{2}}p^{\binom{k-1}{2}}|W_1|\cdots|W_{k-1}|\\
    &=\tfrac{1}{2}(1-\eps)^{k-1}\xi_i^{k-1}3^{-\binom{k}{2}}\EK_k(W_0,\ldots,W_{k-1})
    \greaterByRef{eq:setxiconn}\xi_{i-1}\EK_k(W_0,\ldots,W_{k-1})
  \end{split}\end{equation*}
  bad cliques in $K_k(W_0,\ldots,W_{k-1})$, which is the desired contradiction.

  \begin{claimproof}[Proof of Claim~\ref{cl:normal}]
    We construct the normal $(k-1)$-cliques vertex by vertex in the following
    way. We first construct a set $Z_1\subset W_1$ with
    $|Z_1|\ge\xi_i|W_1|/3$ and then for each $c_1\in Z_1$ a set
    $Z_2(c_1)\subset N_{W_2}(c_1)$ with $|Z_2|\ge\xi_ip|W_2|/9$, and so on, in
    general constructing for $c_1\in Z_1,c_2\in Z_2(c_1),\dots,c_{j-1}\in
    Z_{j-1}(c_1,\dots,c_{j-2})$ a set $Z_j(c_1,\dots,c_{j-1})\subset N_{W_j}(c_1,\dots,c_{j-1})$
    with
    \begin{equation}\label{eq:normal:Z}
      |Z_j(c_1,\dots,c_{j-1})|\ge \xi_i p^{j-1}|W_j|/3^j \,,
    \end{equation}
    where $j$ ranges from~$1$ to~$k-1$, such that the following properties
    hold for each $c_j\in Z_j(c_1,\dots,c_{j-1})$.  Firstly,
    $(c_1,\dots,c_j)$ is in at least
    \begin{equation}\label{eq:normal:bad}
      \xi_i 3^{-j} p^{\binom{k}{2}-\binom{j}{2}}|W_{j+1}|\dots|W_k|
    \end{equation}
    bad $k$-cliques in $K_k(W_1,\dots,W_k)$.  Secondly, for each
    $\ell\in\{0\}\cup\{j+1,\dots,k\}$ the vertex $c_j$ is \emph{typical} with
    respect to $N_{W_\ell}(c_1,\dots,c_{j-1})$, that is,
    \begin{equation}\label{eq:normal:typical}
      |N(c_j)\cap N_{W_\ell}(c_1,\dots,c_{j-1})| = (1\pm\eps)p
      |N_{W_\ell}(c_1,\dots,c_{j-1})|\,.
    \end{equation}
    Observe that by definition $(c_1,\dots,c_j)$ form a clique for each
    $c_1\in Z_1$, $c_2\in Z_2(c_1)$, \dots, $c_{j}\in
    Z_{j}(c_1,\dots,c_{j-1})$.  Moreover, successfully constructing all
    these sets proves Claim~\ref{cl:normal}. Indeed, for each $c_1\in Z_1$,
    $c_2\in Z_2(c_1)$, \dots, $c_{k-1}\in Z_{k-1}(c_1,\dots,c_{k-2})$ the
    clique $(c_1,\dots,c_{k-1})$ satisfies
    $|N_{W_0}(c_1,\dots,c_{k-1})|\ge(1-\eps)^{k-1}p^{k-1}|W_0|$
   by~\eqref{eq:normal:typical}, and $(c_1,\dots,c_{k-1})$ is in at
    least $\xi_i 3^{-k+1} p^{k-1}|W_k|$ bad $k$-cliques in
    $K_k(W_1,\dots,W_k)$ by~\eqref{eq:normal:bad}, so $(c_1,\dots,c_k)$ is
    normal. By~\eqref{eq:normal:Z} there are 
    \[\prod_{j=1}^{k-1}\xi_i3^{-j}p^{j-1}|W_j|=\xi_i^{k-1}3^{-\binom{k}{2}}p^{\binom
      {k-1}{2}}|W_1|\cdots|W_{k-1}| \] such $(k-1)$-cliques
    $(c_1,\dots,c_{k-1})$.

    It remains to show that we can construct sets $Z_j(c_1,\dots,c_{j-1})$
    satisfying~\eqref{eq:normal:Z}, \eqref{eq:normal:bad}
    and~\eqref{eq:normal:typical}. We proceed by induction on $j$, where
    the base case and the inductive step use the same reasoning.  So let
    $\vecc=(c_1,\dots,c_{j-1})$ with $c_1\in Z_1$, \dots, $c_{j-1}\in
    Z_{j-1}(c_1,\dots,c_{j-2})$ be fixed and assume that we constructed
    $Z_1$, \dots, $Z_{j-1}(c_1,\dots,c_{j-2})$ successfully. Now we
    consider $N_{W_j}(\vecc)$ and first bound the size of the set
    $A_j\subset N_{W_j}(\vecc)$ of vertices that
    violate~\eqref{eq:normal:typical} (where, as is usual, we adopt the
    convention that $N_{W_j}(\vecc)=W_j$ if $\vecc$ is empty, which happens
    in the base case $j=1$).
    By~\eqref{eq:normal:typical} in the induction hypothesis we have for
    each $\ell\in\{0\}\cup\{j,\dots,k\}$ that
    \begin{equation}\begin{split}\label{eq:normal:NWl}
        |N_{W_\ell}(\vecc)|
        &\eqByRef{eq:normal:typical}(1\pm\eps)p|N_{W_\ell}(c_1,\dots,c_{j-2})|
        \eqByRef{eq:normal:typical}(1\pm\eps)^2p^2|N_{W_\ell}(c_1,\dots,c_{j-3})| \\
        &\eqByRef{eq:normal:typical} \dots
        \eqByRef{eq:normal:typical}(1\pm\eps)^{j-1}p^{j-1}|W_\ell|
        = (1\pm\eps)^{j-1}p^{j-1}|U_{i-1+\ell}| \\
        &\geByRef{eq:con:Ui:small}
        (1\pm\eps)^{j-1}p^{j-1}\frac{1}{3k}\delta^2\Big(\frac{p}{2}\Big)^{k-i-\ell+2}n
        =
        \frac{(1\pm\eps)^{j-1}\delta^2}{3k\cdot 2^{k-i-\ell+2}} \cdot p^{k-\ell+j-i+1}n
    \end{split}\end{equation}
    and thus
    \begin{equation}\label{eq:normal:NWllarge}
      |N_{W_\ell}(\vecc)|
       \geByRef{eq:setmuepsconn}\sqrt{\mu} p^{k-\ell+j-1}n
        \geByRef{eq:setmuepsconn} \eps p^{2k-2}n \,,
    \end{equation}
  since $i\ge 2$, $j\le k-1$, $\ell\ge 0$. Since~$G$ is
  $(\eps,p,k-1,2k-2)$-pseudorandom it follows from
  Lemma~\ref{lem:pseudorandom} that only $|A_j|<2k\eps p^{k-1}n$
  vertices in~$N_{W_j}(\vecc)$ violate~\eqref{eq:normal:typical}.

  We now construct the desired set~$Z_j(\vecc)$ as follows. We choose among
  the vertices in $N_{W_j}(\vecc)\setminus A_j$ those $\xi_i
  p^{j-1}|W_j|/3^j$ vertices $c_j$ which are together with $\vecc$ in the
  biggest number of bad $k$-cliques in $K_k(W_1,\dots,W_k)$. This is
  possible since 
  \begin{equation}\label{eq:normal:W}
    |W_j|\geByRef{eq:con:Ui:small} \frac1{3k}\delta^2(p/2)^{k-i-j+2}n
    \ge100k\sqrt{\mu} p^{k-j} n
    \ge100k\eps p^{k-j} n
  \end{equation}
  because $i\ge 2$, and by~\eqref{eq:normal:NWl} we have
  $|N_{W_j}(\vecc)|\ge(1-\eps)^{j-1}p^{j-1}|W_j|\ge\frac34p^{j-1}|W_j|$ and
  hence $|N_{W_j}(\vecc)\setminus A_j|\ge \frac34p^{j-1}|W_j|-2k\eps
  p^{k-1}n\ge\frac12 p^{j-1}|W_j|$. By construction $Z_j(\vecc)$
  satisfies~\eqref{eq:normal:Z} and~\eqref{eq:normal:typical}. In the
  remainder of this proof we will show that~$Z_j(\vecc)$ also
  satisfies~\eqref{eq:normal:bad}.

  For this purpose we next estimate how many $k$-cliques in $K_k(W_1,\dots,W_k)$ use
  $\tpl{c}$ and a vertex in~$A_j$ (which is an upper bound on the number of
  bad $k$-cliques that we ``lose'' to~$A_j$). Observe that
  $\big|K_{k-j+1}\big(A_j,N_{W_{j+1}}(\vecc),\ldots,N_{W_k}(\vecc)\big)\big|$
  is exactly the number of such $k$-cliques.  In order to upper bound this
  quantity we want to apply Proposition~\ref{prop:cliquecount} and so we
  must justify that the sets $A_j,N_{W_{j+1}}(\vecc),\ldots,N_{W_k}(\vecc)$
  are large enough for this application. In fact, $|A_j|$ is not large
  enough, but we can rectify this by adding arbitrary vertices of~$N_{W_j}(\vecc)$ to
  obtain a set $A'_j$ of size 
  $\sqrt{\mu}|N_{W_j}(\vecc)|\ge \mu p^{k-1} n$, where we used~\eqref{eq:normal:NWllarge}. 
 By~\eqref{eq:normal:NWllarge} we also have
  $|N_{W_\ell}(\vecc)|\ge \mu p^{k-\ell+j-1}n$.  Thus
  Proposition~\ref{prop:cliquecount} implies that at most 
    \begin{equation*}\begin{split}
      \big|K_{k-j+1}\big(A_j,&N_{W_{j+1}}(\vecc),\ldots,N_{W_k}(\vecc)\big)\big|
      \le \big|K_{k-j+1}\big(A'_j,N_{W_{j+1}}(\vecc),\ldots,N_{W_k}(\vecc)\big)\big|\\
      &\le
      (1+\mu)\EK_{k-j+1}\big(A'_j,N_{W_{j+1}}(\vecc),\ldots,N_{W_k}(\vecc)\big)\\
      &=(1+\mu) p^{\binom{k-j+1}{2}}\sqrt{\mu}|N_{W_j}(\vecc)||N_{W_{j+1}}(\vecc)|\cdots|N_{W_k}(\vecc)| \\
      & \leByRef{eq:normal:NWl} 
      (1+\mu) p^{\binom{k-j+1}{2}}\sqrt{\mu}
      (1+\eps)^{(k-j+1)(j-1)}p^{(k-j+1)(j-1)}|W_j|\cdots|W_k| \\
      &\leByRef{eq:setmuepsconn}
      \tfrac{1}{6}\xi_i3^{1-j}p^{\binom{k}{2}-\binom{j-1}{2}}|W_j|\cdots|W_k|\,.
    \end{split}\end{equation*}
   $k$-cliques in $K_k(W_1,\dots,W_k)$ use $\tpl{c}$ and a vertex in~$A_j$.

   From this together with~\eqref{eq:normal:bad} in the induction
   hypothesis we immediately get that the number of bad cliques in
   $K_k(W_1,\dots,W_k)$ which use $\tpl{c}$ and a vertex in
   $N_{W_j}(\vecc)\setminus A_j$ is at least
    $\tfrac{5}{6}\xi_i3^{1-j}
    p^{\binom{k}{2}-\binom{j-1}{2}}|W_j|\cdots|W_k|$.
    We claim (and show below) that moreover at most half of these, i.e., at most
    \begin{equation}\label{eq:normal:badZ}
      \tfrac{5}{12}\xi_i3^{1-j}
      p^{\binom{k}{2}-\binom{j-1}{2}}|W_j|\cdots|W_k|
    \end{equation}
    bad cliques in $K_k(W_1,\dots,W_k)$ use $\tpl{c}$ and a vertex
    in~$Z_j(\vecc)$. Hence the vertex $c_j$ in $N_{W_j}\big(\vecc)\setminus(A_j\cup
    Z_j(\vecc)\big)$ which together with $\vecc$ is in the biggest number of bad $k$-cliques 
    in $K_k(W_1,\dots,W_k)$ is in at least 
    \[\tfrac{5}{12}\xi_i3^{1-j}
    p^{\binom{k}{2}-\binom{j-1}{2}}|W_j|\cdots|W_k|/\big|N_{W_j}(\tpl{c})\big|\geByRef{eq:normal:NWl} \xi_i3^{-j
    }p^{\binom{k}{2}-\binom{j}{2}}|W_{j+1}|\cdots|W_k|\]
    such bad $k$-cliques. By construction of $Z_j(\vecc)$  this is thus also
    true for the vertices~$c_j$ in $Z_j(\vecc)$ and hence we get~\eqref{eq:normal:bad}.

    It remains to establish~\eqref{eq:normal:badZ}, which we obtain 
    (similar as before) by bounding the size of $K_{k-j+1}\big(Z_j(\vecc),N_{W_{j+1}}(\vecc),\ldots,N_{W_k}(\vecc)\big)$
    with the help of Proposition~\ref{prop:cliquecount}. Indeed,
    by~\eqref{eq:normal:W} and~\eqref{eq:setmuepsconn} we have $|Z_j(\vecc)|=\xi_i p^{j-1}|W_j|/3^j
    \ge (\xi_i/3^j) 100k\sqrt{\mu} p^{k-1} n \ge\mu p^{k-1}n$ and 
    hence this proposition implies
   \begin{equation*}\begin{split}
      \big|K&_{k-j+1}\big(Z_j(\vecc),N_{W_{j+1}}(\vecc),\ldots,N_{W_k}(\vecc)\big)\big|\\
      &\le
      (1+\mu)p^{\binom{k-j+1}{2}}|Z_j(\vecc)||N_{W_{j+1}}(\vecc)|\cdots|N_{W_k}(\vecc)| \\
      &\leByRef{eq:normal:NWl} (1+\mu)p^{\binom{k-j+1}{2}}\cdot \xi_i3^{-j}p^{j-1}
      |W_j|\cdot (1+\eps)^{(j-1)(k-j)}p^{(j-1)(k-j)}|W_{j+1}|\cdots|W_k|\\
      &\leByRef{eq:setmuepsconn} \tfrac{5}{12}\xi_i3^{1-j}
      p^{\binom{k}{2}-\binom{j-1}{2}}|W_j|\cdots|W_k|
    \end{split}\end{equation*}
    as desired.
  \end{claimproof}
  This concludes the proof of Claim~\ref{clm:conn:indhyp}.
\end{proof}

\section{Proof of Lemma~\ref{lem:reslem} and
  Lemma~\ref{lem:covlem}}\label{sec:reservoir}

In this section we will prove the following technical lemma, which 
implies both Lemma~\ref{lem:reslem} and Lemma~\ref{lem:covlem}. 

\begin{lemma}\label{lem:rescovlem}
Given $k\ge 2$, $0<\delta<1/4$ 
and $0<\beta<1/2$ there exists an  $\eps>0$ 
such that the following holds. 
Let $0<p<1$ and let $G$ be an $n$-vertex graph.
Suppose that $G$ is $(\eps,p,1,2)$-pseudorandom if $k=2$; and
$(\eps,p,k-1,2k-1)$-pseudorandom and $(\eps,p,k,k+1)$-pseudorandom if
$k>2$.  Let $R$ and $S$ be disjoint subsets of $V(G)$ with $|R|\le\delta
n/(200k)$ and $|S|\ge\delta n$ such that $\deg_{S}(v)\ge \beta\delta pn/2$
for all $v\in R$.
Then there is a $k$-path $P$ contained in $R\cup
S$ with the following properties.
\begin{enumerate}[label=\abc]
 \item\label{rescovlem:R} $R\subset V(P)$ and $|V(P)|\le 50k|R|$.
 \item\label{rescovlem:W} For any $W\subset R$, there is a $k$-path on
$V(P)\setminus W$ whose start and end $k$-tuples are identical to those of $P$.
 \item \label{rescovlem:S}  The start and end $k$-tuples of the $k$-path $P$ are in $S$ and are
$(\tfrac18,p)$-connected to $S\setminus V(P)$.
 \item\label{rescovlem:conn} If $|R|\ge\delta^2 n/(200k)$, then the start and end
 $k$-tuples of $P$ are
$(\tfrac12,p)$-connected to $R$.
\end{enumerate}
\end{lemma}

In order to infer the Reservoir lemma, Lemma~\ref{lem:reslem}, from this
lemma use $R$ as given and set $S:=V(G)\setminus R$. For the Covering
lemma, Lemma~\ref{lem:covlem} use $S$ as given and set $R:=L$.

For the proof of Lemma~\ref{lem:rescovlem} we use the following definition.
Given $k$, a \emph{$k$-reservoir graph} is a graph which contains a
spanning $k$-path $P$ with the following extra property. There is a special
vertex $r$, which we call the \emph{reservoir vertex} of the reservoir
graph, such that $V(P)\setminus\{r\}$ forms a $k$-path whose start and end
$k$-tuples are identical to those of $P$. We also call these tuples the
\emph{start and end tuple of the reservoir graph}.  To give a simple
example, the triangle $abc$ is a $1$-reservoir graph, with $P=(a,c,b)$ and
$c$ being the reservoir vertex, and~$K_5$ is a $2$-reservoir
graph. However, in our proof of Lemma~\ref{lem:rescovlem} we shall need
much sparser reservoir graphs. The following lemma states that such graphs
exist.

\begin{lemma}[Reservoir graph lemma]\label{lem:resgraph}
  For all $k\ge 2$, $\beta>0$ and $0<\delta<1/4$ there exists an $\eps>0$
  such that the following holds. Let $G$ be an $n$-vertex graph, $S$ a
  subset of $V(G)$ of size at least $\delta n/2$, and $R^*$ a subset of
  $V(G)$ of size at least $\delta^2 n/(200k)$. Let $r$ be a vertex of
  $V(G)\setminus S$ with at least $\beta \delta pn/8$ neighbours in
  $S$. For $k\geq 3$ suppose that $G$ is $(\eps,p,k,k+1)$-pseudorandom and
  for $k= 2$ suppose that $G$ is $(\eps,p,1,2)$-pseudorandom.

  Then there is a reservoir graph
  $H$ in $G$ whose reservoir vertex is $r$, whose remaining vertices are in $S$,
  and whose start and end $k$-tuple are $(\tfrac12,p)$-connected to $S$ and
  to $R^*$. Furthermore the reservoir graph has at most $\max(47,2k+1)$ vertices.
\end{lemma}

Observe that this lemma allows us to specify the reservoir vertex~$r$ and a
small vertex set~$S$ which contains the remaining constant number of
vertices of the reservoir graph -- the only requirement being that~$r$ has
many neighbours in~$S$.  Moreover, the lemma guarantees well-connectedness
of the start and end tuple to~$S$ and an additional vertex set~$R^*$, which
we shall use to obtain property~\ref{rescovlem:conn} in the proof of
Lemma~\ref{lem:rescovlem}.

We prove Lemma~\ref{lem:resgraph} at the end of this section and next show
how it entails Lemma~\ref{lem:rescovlem}. We first briefly explain the
idea. Roughly speaking, our goal is to construct for each $r\in R$ a
reservoir graph which uses~$r$ as reservoir vertex and to connect up all
these reservoir graphs into a long path. Observe that the definition of a
reservoir graph implies that the resulting long path is a reservoir path,
that is, a $k$-path satisfying~\ref{rescovlem:W} of
Lemma~\ref{lem:rescovlem}.

\begin{proof}[Proof of Lemma~\ref{lem:rescovlem}]
  Let $k\ge 2$, $\delta$ and $\beta$ satisfy the conditions given in
  the statement of Lemma~\ref{lem:rescovlem}. We require $\eps>0$ to be
  sufficiently small to apply Lemma~\ref{lem:resgraph} with input $k$, $\beta$ and $\delta$,
  to apply Lemma~\ref{lem:connect} with input $k$ and $\delta/4$, and smaller
  than $\beta\delta/(200k)$.
  
  Let $p$, $G$, $R$ and $S$ be as in the statement of
  Lemma~\ref{lem:rescovlem}. Our approach now is as follows.  We want to
  choose one vertex $r_1$ of $R$ and use the Reservoir graph lemma,
  Lemma~\ref{lem:resgraph}, to construct a $k$-reservoir graph $H_1$ with
  reservoir vertex $r_1$ and all other vertices in~$S$. We then want to
  continue by choosing a second vertex $r_2$ of $R$, and repeat this
  procedure to obtain $H_2$, avoiding the vertices of $H_1$. Next we want
  to use the Connection lemma, Lemma~\ref{lem:connect}, to connect the end
  tuple of $H_1$ to the start tuple of $H_2$, again within $S$. We will
  repeat this until finally we construct $H_{|R|}$ and connect it to
  $H_{|R|-1}$, and the result is the desired reservoir path $P$.

  Before we can start we have to set up~$R^*$ for
  Lemma~\ref{lem:rescovlem}. Recall that the purpose of~$R^*$ will be to
  guarantee property~\ref{rescovlem:conn}.  So, if we have $|R|\ge\delta^2
  n/(200k)$, we set $R^*:=R$; otherwise we set $R^*:=V(G)$.

  We now perform the first step of our procedure. Let
  $r_1$ be the vertex of $R$ of lowest degree to~$S$. By our assumptions we
  have $\deg_S(r_1)\ge\beta\delta p n/2$, hence we can apply
  Lemma~\ref{lem:resgraph}. Let~$H_1$ be the $k$-reservoir
  graph with reservoir vertex $r_1$ and remaining vertices in $S$
  guaranteed by this lemma. The start $k$-tuple of $H_1$ is
  $(\tfrac12,p)$-connected to $S$, and so by the choice of $\eps$, and
  Remark~\ref{rem:pseudo}, it is also $(\tfrac18,p)$-connected to $S\setminus
  V(H_1)$. In the following steps of our procedure, whose goal is to
  construct the reservoir path~$P$, we have to be careful to
  avoid destroying this connectedness of the start $k$-tuple to~$S\setminus
  P$ (in order to obtain
  Property~\ref{rescovlem:S} of Lemma~\ref{lem:rescovlem}). Hence we
  shall now fix witnesses of this connectedness and avoid using these
  vertices in the following.
  So let $Z\subset S\setminus V(H_1)$ be a vertex set that witnesses that 
  the start $k$-tuple $(x_1,\dots,x_k)$ of~$H_1$ is $(\tfrac18,p)$-connected to $S\setminus
  V(H_1)$ with $|Z|\le\frac18(p/2)|S|$, which is possible by
  Remark~\ref{rem:witness}.     
  We call the vertices in $Z$ and $V(H_1)$ \emph{used}, and all other
  vertices of $R\cup S$ \emph{unused}.
  
  Now for each $2\le i\le |R|$ in succession, we perform the following
  procedure. Let $S'$ be the unused vertices of $S$. Let $r_i$ be an unused
  vertex of $R$ with fewest neighbours in $S'$. We claim (and justify
  below) that $r_i$ has at least $\beta\delta p n/8$ neighbours in $S'$,
  and that $|S'|\ge\delta n/2+100k$. On this assumption, we can apply
  Lemma~\ref{lem:resgraph} to obtain a $k$-reservoir graph~$H_i$ with
  reservoir vertex~$r_i$ and remaining vertices in~$S'$. By construction
  the end tuple of~$H_{i-1}$ is $(\tfrac12,p)$-connected to~$S'$ and $R^*$, and hence
  by the choice of~$\eps$ also $(\tfrac14,p)$-connected to $S'\setminus
  V(H_i)$, as is the start tuple of $H_i$. Since $\big|S'\setminus
  V(H_i)\big|\ge\delta n/2$, we can apply Lemma~\ref{lem:connect} to find a
  connecting $k$-path of length at most~$7k$ from the end tuple of
  $H_{i-1}$ to the start tuple of~$H_i$ whose remaining vertices are in
  $S'\setminus V(H_i)$. We mark this $k$-path, $r_i$ and $V(H_i)$ as used.
  
  Assuming we successfully complete the above procedure to obtain a
  $k$-path $P$, Lemma~\ref{lem:rescovlem} follows.  Indeed, the path~$P$
  then covers all vertices of $R$ and certainly uses at most $50k|R|$
  vertices, since the $k$-reservoir graph and connecting $k$-path that we
  construct at each step (except the first, where we only construct the
  $k$-reservoir graph) contain at most $50k$ vertices, hence we obtain
  property~\ref{rescovlem:R} of
  Lemma~\ref{lem:rescovlem}. Property~\ref{rescovlem:W} follows from the
  definition of a $k$-reservoir graph and the fact that we created for each
  vertex $r\in R$ one of these reservoir graphs with reservoir vertex~$r$
  and connected them to form~$P$.
  Property~\ref{rescovlem:S} follows
  by observing that $Z\cap V(P)=\emptyset$, witnessing that the start
  $k$-tuple of~$P$ is $(\frac18,p)$-connected to $S\setminus V(P)$.
  Moreover,
  the end tuple of $P$ is the end tuple of $H_{|R|}$, which by construction
  and choice of $\eps$ is $(\tfrac18,p)$-connected to $S\setminus
  V(P)$. Finally we have property~\ref{rescovlem:conn} because each~$H_i$
  is $(\frac12,p)$-connected to~$R^*$.
  
  It remains only to justify our assumptions that $r_i$ has at least
  $\beta\delta p n/8$ neighbours in $S'$ and that $|S'|\ge\delta n/2+23k$
  at each step $i$. The latter clearly follows from the facts that
  $|S|\ge\delta n$, that $|P|\le 50k|R| \le\delta n/4$, that
  $|Z|\le\frac18(p/2)|S|$ and that $n$ is sufficiently large due to the choice of~$\eps$ and Remark~\ref{rem:pseudo}. For the former, suppose
  for contradiction that at step~$i$ in the above procedure, we find that
  the vertex $r_i$ has less than $\beta\delta p n/8$ unused neighbours in
  $S'$. Since $r_i$ has at least $\beta\delta p n/2$ neighbours in $S$, it
  follows that 
  $i\ge\tfrac{3}{8}\beta\delta p n/\max(47,2k+1)>\beta\delta p n/(100k)$. Furthermore, at the step
  $i':=i-\beta\delta p n/(100k)+1$, certainly $r_i$ had at most
  $\beta\delta p n/8+50k(i-i')<\beta\delta p n/2$ unused neighbours, since
  only $50k(i'-i)$ vertices were used in the intervening steps. As the
  vertex $r_i$ was not chosen at any of the steps $i'\le j<i$, it follows
  that at each step $r_j$ had at most as many unused neighbours as $r_i$,
  and in particular less than $\beta\delta p n/2$ unused neighbours. Let
  $S'$ be the set of unused vertices in $S$ at step~$i$. We conclude that
  each of the vertices $Q:=\{r_{i'},\ldots,r_i\}$ has less than
  $\beta\delta p n/2<\delta p n/4$ neighbours in $S'$, and thus $e(Q,S')\le
  \delta p n |Q|/4<(1-\eps)p|Q||S|$. Since $|S|\ge \delta n/2>\eps n$, and
  $|Q|=i-i'=\beta\delta p n/(100k)>\eps p n$, and since $G$ is
  $(\eps,p,0,1)$-pseudorandom, this is a contradiction
  to~\eqref{eq:pseudodisc}.  It follows that our assumptions are justified,
  which completes the proof.
\end{proof}

It remains to prove Lemma~\ref{lem:resgraph}.  Again we split the proof
into the two cases $k=2$, and $k\ge 3$. In the case $k\ge3$ a rather
straightforward construction works, which generalises the example of the triangle
we gave above: Our $k$-reservoir graph has
$2k+1$ vertices, and consists simply of a $2k$-vertex $k$-path in $S$
all of whose vertices are adjacent to~$r$.

\begin{proof}[Proof of Lemma~\ref{lem:resgraph} for $k\ge3$]
  We set $\eps=\beta\delta^2/(1600k^2)$.  
  Our goal is to construct a $2k$-vertex $k$-path in $S\cap
  N(r)$. Obviously, such a $k$-path together with~$r$ forms a $k$-reservoir
  graph with reservoir vertex~$r$. So let $X_1,\ldots,X_{2k}$ be any
  collection of pairwise disjoint subsets of $S\cap N(r)$, each of size
  $\beta\delta pn/(16k)$, which we can choose because $|N(r)\cap
  S|\ge\beta\delta pn/8$.  We now construct the desired $k$-path by
  choosing one vertex from each of the~$X_i$.  However, we have to bear in
  mind that we also want the start and the end $k$-tuple of this path to be
  $(\frac12,p)$-connected to both~$R^*$ and~$S$ (which is why we impose
  conditions \eqref{eq:resgraph:typ1}--\eqref{eq:resgraph:typ4} below).

 We shall call a
  vertex~$x$ \emph{typical} with respect to a set~$Y$ if $\big|N(x)\cap
  Y\big|\ge(1-\eps)p|Y|$, and atypical otherwise.  Since~$G$ is $(\eps,p,k,k+1)$-pseudorandom
  Lemma~\ref{lem:pseudorandom} implies that for each $|Y|\ge \eps p^k n$
  less than $\eps p^{k+1} n$ vertices of~$G$ are atypical with respect
  to~$Y$.

  Firstly, for each $1\le i\le k$, in
  order, we choose any $x_i\in X_i\cap N(x_1,\dots,x_{i-1})$ which is
  typical with respect to each of the following at most~$3k$ sets:
  \begin{align}
    &R^*,R^*\cap N(x_1),\ldots,R^*\cap N(x_1,\ldots,x_{i-1})\,, \label{eq:resgraph:typ1}\\
    &S,S\cap N(x_1),\ldots,S\cap N(x_1,\ldots,x_{i-1})\,,\\
    &X_{i+1}\cap N(x_1,\ldots,x_{i-1}),\ldots,X_{k+1}\cap
    N(x_1,\ldots,x_{i-1})\quad \text{and} \notag\\
    &X_{k+2}\cap N(x_2,\ldots,x_{i-1}),\ldots,X_{k+i-1}\cap
    N(x_{i-1}),X_{k+i}\,. \notag
  \end{align}
  Secondly, for each $k+1\le i\le 2k$, in order, we choose an $x_i\in X_i\cap
  N(x_{i-k},\ldots,x_{i-1})$ which is typical with respect to each of the
  following at most~$3k$ sets:
  \begin{align}
    &R^*,R^*\cap N(x_{k+1}),\ldots,R^*\cap N(x_{k+1},\ldots,x_{i-1})\,,\\
    &S,S\cap N(x_{k+1}),\ldots,S\cap
    N(x_{k+1},\ldots,x_{i-1})\quad\text{and} \label{eq:resgraph:typ4}\\
    &X_{i+1}\cap N(x_{i-k},\ldots,x_{i-1}),\ldots,X_{2k}\cap 
    N(x_{i-k},\ldots,x_{i-1})\,. \notag
  \end{align}
  Clearly vertices $x_1,\dots,x_{2k}$ chosen in this way form a $k$-path.
  Moreover, choosing these vertices is possible for the following reasons.  Since
  $|X_i|\ge\beta\delta pn/(16 k)$ and $|S|\ge|R^*|\ge\delta^2n/(200k)$ and
  by typicality in earlier steps, the smallest sets to which we require
  typicality are $X_\ell\cap N(x_1,\dots,x_{k-1})$ and $X_\ell\cap
  N(x_{i-k},\dots,x_{i-1})$ for certain values of~$\ell$. Each of these
  sets involve the joint neighbourhood of $k-1$ vertices in one of
  the~$X_\ell$, hence (by typicality in earlier steps) these sets are of
  size at least
  $(1-\eps)^{k-1}p^{k-1}|X_\ell|\ge(1-\eps)^{k-1}p^k\beta\delta n/(16
  k)\ge\eps p^k n$. Thus none of the sets to which we require typicality
  are smaller than $\eps p^k n$. In addition, there are at most $3k$ sets to
  which we require typicality, forbidding at most $3k\eps p^{k+1}n$
  vertices for the choice of~$x_i$, out of
  $\big((1-\eps)p\big)^k|X_i|\ge(1-\eps)^k\beta\delta^2
  p^{k+1}n/(200k)>4k\eps p^{k+1}n$ vertices.

  In order to show that $(x_k,\ldots,x_1)$ is moreover
  $(\tfrac12,p)$-connected to $R^*$ we need to check that
  $\deg_{R^*}(x_j,\dots,x_1)\ge\frac12(\frac p2)^{j}|R^*|$ for each
  $j\in[k]$. Indeed, it follows from~\eqref{eq:resgraph:typ1} that
  $\deg_{R^*}(x_j,\dots,x_1)\ge(1-\eps)^j p^j|R^*|>\frac12p^j|R^*|$.
  Similarly, $(x_k,\ldots,x_1)$ is $(\tfrac12,p)$-connected to~$S$, and
  $(x_{k+1},\ldots,x_{2k})$ is $(\tfrac12,p)$-connected to both~$R^*$
  and~$S$.
\end{proof}

In the case $k=2$ we need to work with weaker pseudorandomness conditions
and thus use a more involved construction, illustrated in Figure~\ref{fig:reservoir2}.
In this construction there is a $2$-path from $a_1a_2$ to $b_7b_8$ using all vertices in the
left-to-right order, and a second, using all vertices but the reservoir vertex
$r$, which starts $a_1a_2a_3a_4$ then goes ``left'' to $b_1b_2b_3b_4$, and so
on to finish $b_6b_5b_7b_8$.

We will call the subgraph of the $2$-reservoir graph induced by
$\{a_1,\dots,a_8\}\cup\{b_1,\dots,b_8\}\cup\{r\}$ the \emph{spine} of the
reservoir graph (that is, the graph on the large vertices and with the thick edges in
Figure~\ref{fig:reservoir2}). The construction of this graph uses similar
ideas and a similar pigeonhole argument as the construction of a
connecting $2$-path for the Connection lemma.

\begin{figure}
 \caption{The $2$-reservoir graph}
 \psfrag{a1}{$a_1$}
 \psfrag{a2}{$a_2$}
 \psfrag{a3}{$a_3$}
 \psfrag{a4}{$a_4$}
 \psfrag{a5}{$a_5$}
 \psfrag{a6}{$a_6$}
 \psfrag{a7}{$a_7$}
 \psfrag{a8}{$a_8$}
 \psfrag{r}{$r$}
 \psfrag{b1}{$b_1$}
 \psfrag{b2}{$b_2$}
 \psfrag{b3}{$b_3$}
 \psfrag{b4}{$b_4$}
 \psfrag{b5}{$b_5$}
 \psfrag{b6}{$b_6$}
 \psfrag{b7}{$b_7$}
 \psfrag{b8}{$b_8$}
 \includegraphics[width=12cm]{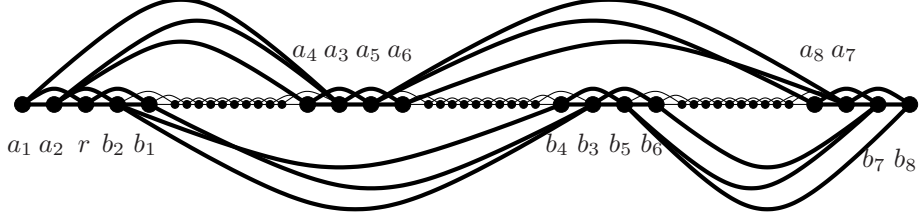}
 \label{fig:reservoir2}
\end{figure}

\begin{proof}[Proof of Lemma~\ref{lem:resgraph} for $k=2$]
  We let $\eps\le\beta^2\delta^2/10^6$ be small enough so that we can apply
  Lemma~\ref{lem:connect} with input $k=2$ and $\delta/4$. Let~$G$, $S$,
  $R^*$ and $r\in V(G)\setminus S$ with
  \begin{equation*}
    |S|\ge\frac{\delta n}{2}\,, \quad
    |R^*|\ge\frac{\delta^2 n}{400}\,, \quad
    |N_S(r)|\ge\frac{\beta\delta pn}8
  \end{equation*}
  be given.  In this proof we call a vertex~$x$ \emph{typical} with respect
  to a set~$Y$ if $\big|N_Y(x)\big|\ge(1-\eps)p|Y|$.

  Our goal is to find in~$S$ a copy of the $2$-reservoir graph depicted in
  Figure~\ref{fig:reservoir2}.  In a first step we now construct the
  five-vertex $2$-path $(a_1,a_2,r,b_2,b_1)$, sufficiently well-connected
  to~$S$ and~$R^*$.  For this we choose distinct vertices in the following
  order:
  \begin{align*}
     a_1 &\in N_S(r) &\text{typical with respect to }& S \,, R^* \,, N_S(r) \,,\\
     a_2 &\in N_S(r,a_1) &\text{typical with respect to }&S\,, R^*\,, N_S(r)\,, N_S(a_1)\,, 
     N_{R^*}(a_1) \,, \\
     b_2 &\in N_S(r,a_2) &\text{typical with respect to }&S\,, N_S(r)\,,\\
     b_1 &\in N_S(r,b_2) &\text{typical with respect to }&S\,, N_S(b_2)\,.
  \end{align*}
  This is possible by $(\eps,p,1,2)$-pseudorandomness of $G$, since in each case
  we choose from a set of vertices of size at least 
  \begin{equation*}
    \min \big\{|N_S(r,a_1)|,|N_S(r,a_2)|,|N_S(r,b_2)|\big\}
    \ge (1-\eps)\tfrac18\beta\delta p^2 n-3>\eps p^2 n\,,
  \end{equation*}
  by typicality in earlier choices. Moreover, we require typical behaviour
  only to at most five sets each of size at least $(1-\eps)\beta\delta^2
  pn/400>\eps p n$, where the left hand side of this calculation is given
  by the lower bounds on~$|N_S(r)|$ and~$|N_{R^*}(a_1)|$.
  Now it is easy to check that
  $(a_1,a_2,r,b_2,b_1)$ forms a $2$-path in $G$, that $(a_1,a_2)$,
  $(a_2,a_1)$, $(b_1,b_2)$ and $(b_2,b_1)$ are $(\tfrac12,p)$-connected to
  $S$, and $(a_2,a_1)$ is $(\tfrac12,p)$-connected to $R^*$.
  
  Our second step is to construct the remainder of the spine of the
  reservoir graph. Let us first investigate the structure of this graph
  more closely and describe our strategy. First we will construct
  candidates for the induced subgraph of the spine on
  $a_1,\dots,a_8$.  Note that each edge of this subgraph is contained in one
  of the three $2$-paths
  \begin{equation*}
    (a_1,a_2,a_3,a_4)\,, \quad (a_4,a_3,a_5,a_6)\,, \quad (a_6,a_5,a_7,a_8)\,.
  \end{equation*}
  Then we will construct candidates for the induced subgraph on $b_1,\ldots,b_8$.
  For this observe that the adjacencies among the vertices $a_1,\ldots,a_8$ and
  those among $b_1,\ldots,b_8$ are identical. Finally, in order to complete
  the spine, we need
  only in addition to guarantee that $(a_8,a_7,b_7,b_8)$ is a $2$-path. We will
  show that we can choose vertices among our candidates so that this is
  satisfied. When choosing the various candidates we have to keep in mind
  that we will want $(b_7,b_8)$ to be $(\frac12,p)$-connected to~$S$
  and~$R^*$ as required in the conclusion of our lemma. In addition we
  want to complete the constructed spine to obtain the whole reservoir
  graph in~$S$ with the help of the Connection lemma. Hence we will also
  need that
  \begin{equation}\label{eq:res:conn}
  \begin{alignedat}{3}
    &(a_3,a_4)\,,&\quad &(a_5,a_6)\,,&\quad &(a_7,a_8) \\
    &(b_3,b_4)\,,&&(b_5,b_6)\,,&&(b_7,b_8) 
  \end{alignedat}
  \end{equation}
  are $(\frac12,p)$-connected to~$S$.

  We now first define in~$S$ twelve disjoint parts
  $X_3,\dots,X_8,X'_3,\dots,X'_8$. We will then find candidates (with the
  required properties) for~$a_i$ in~$X_i$ and for~$b_i$ in~$X'_i$,
  $i\in\{3,\dots,8\}$.  So choose
  \begin{align*}
    X_3&\subset N_S(a_1,a_2)\setminus\{a_1,a_2,b_1,b_2\} &
    \quad\text{with}\quad |X_3|&=p^2|S|/20\,, \\
    X'_3&\subset N_S(b_1,b_2)\setminus\big(\{a_1,a_2,b_1,b_2\}\cup
    X_3\big) &
    \quad\text{with}\quad |X'_3|&=p^2|S|/20\,, \\
    X_4&\subset N_S(a_2)\setminus\big(\{a_1,a_2,b_1,b_2\}\cup X_3\cup
    X'_3\big) &
    \quad\text{with}\quad |X_4|&=p|S|/20\,, \\
    X'_4&\subset N_S(b_2)\setminus\big(\{a_1,a_2,b_1,b_2\}\cup X_3\cup
    X'_3\cup X_4\big) &
    \quad\text{with}\quad |X'_4|&=p|S|/20\,.
  \end{align*}
  Note that these sets exist because $(a_1,a_2)$
  and $(b_1,b_2)$ are $(\tfrac12,p)$-connected to $S$.  Further, choose
  \begin{equation*}
    X_5,\dots,X_8,X'_5,\dots,X'_8 \subset
    S\setminus\big(\{a_1,a_2,b_1,b_2\}\cup X_3\cup X_4\cup X'_3\cup
    X'_4\big)
  \end{equation*}
  each of size $|S|/20$ and pairwise disjoint.
  
  Let us now turn to the candidate sets for the vertices~$a_i$. The
  construction of these sets is somewhat technically intricate, as we will proceed
  differently for different vertices~$a_i$. It might help the reader to
  keep in mind though that in each instance the main purpose of the
  definition of a candidate set is to guarantee the necessary adjacencies
  and well-connectedness.
 
  We start with the candidate set for~$a_3$. For its choice observe that
  $a_3$ is adjacent to~$a_4$, $a_5$ and $a_6$ and recall that we want
  $(a_3,a_4)$ to be $(\frac12,p)$-connected to~$S$.  Thus let $A_3\subset
  X_3$ be those vertices with at least $p|X_4|/2=p^2|S|/40$ neighbours in
  $X_4$, at least $p|X_5|/2=p|X_6|/2=p|S|/40$ neighbours in each of $X_5$
  and $X_6$, and at least $p|S|/2$ neighbours in $S$.  By
  Lemma~\ref{lem:pseudorandom} we have $|A_3|\ge|X_3|-4\eps p^2 n>0$. We
  fix a vertex $a_3\in A_3$. By the definition of~$X_3$ this~$a_3$ is
  adjacent to~$a_1$ and~$a_2$ as required.
  
  We next define the candidate set for~$a_4$.  We will choose this
  candidate set in $N_{X_4}(a_3)$, so again by definition, all candidates
  for~$a_4$ will be adjacent to~$a_2$ and~$a_3$ as required. Moreover,
  $a_4$ should be adjacent jointly with~$a_3$ to~$a_5$, and $(a_3,a_4)$
  should be well-connected to~$S$. So let $A_4\subset N_{X_4}(a_3)$ be
  those vertices with at least $p^2|S|/80$ neighbours
  in~$N_{X_5}(a_3)$,
  with at least $p|S|/2$ neighbours in $S$, and at least $p^2|S|/4$
  neighbours in $N_S(a_3)$.  By the definition of~$A_3$ we have
  $|N_{X_5}(a_3)|\ge p|S|/40$ and
  $|N_S(a_3)|\ge p|S|/2$, and therefore Lemma~\ref{lem:pseudorandom} implies
  \begin{equation}\label{eq:res:Y4}
    |A_4|\ge |N_{X_4}(a_3)|-3\eps p^2 n
    \ge p^2|S|/40-3\eps p^2 n\ge p^2|S|/80\,,
  \end{equation}
  where the second inequality follows from the definition of~$A_3$.
  We do not choose~$a_4$ from this candidate set immediately, but first
  define further candidate sets.

  We will choose the candidate set for~$a_5$ in $N_{X_5}(a_3)$ and we will
  select only vertices which have some neighbour in~$A_4$. This will guarantee
  that~$a_5$ is connected to~$a_3$ as required and that, once we choose a
  vertex~$a_5$ we can also choose a valid vertex~$a_4$. In addition~$a_5$
  should be adjacent jointly with~$a_3$ to $a_6$, adjacent to~$a_7$
  and~$a_8$, and $(a_5,a_6)$ should be $(\frac12,p)$-connected to~$S$.
  So let $A_5\subset N_{X_5}(a_3)$ be those vertices with a neighbour in
  $A_4$, with at least $p^2|S|/80$ neighbours in $N_{X_6}(a_3)$, with at
  least $p|X_7|/2=p|X_8|/2=p|S|/40$ neighbours in both $X_7$ and $X_8$, and
  with at least $p|S|/2$ neighbours in~$S$. 
  Observe that by~\eqref{eq:res:Y4} the set~$A_4$ is of size $\Omega(p^2
  n)$ only, and so Lemma~\ref{lem:pseudorandom} only guarantees that at
  most $\eps p n$ vertices in $N_{X_5}(a_3)$ do not have a neighbour
  in~$A_4$. The remaining requirements for~$A_5$ however involve only sets
  of size $\Omega(pn)$ because $|N_{X_6}(a_3)|\ge p|S|/40$ by the
  definition of~$A_3$ and so Lemma~\ref{lem:pseudorandom} implies
  \begin{equation*}
    |A_5|\ge |N_{X_5}(a_3)|-\eps p n-4\eps p^2 n
    \ge p|S|/40-\eps p n-4\eps p^2 n\ge p|S|/80\,,
  \end{equation*}
  where the second inequality follows from the definition of~$A_3$.
  Again, we will choose~$a_5$ later, but next define for each possible
  choice of $a_5\in A_5$ candidate sets $A_6(a_5)$ and $A_7(a_5)$ for~$a_6$
  and~$a_7$. 

  For each $a_5\in A_5$ we will now define a candidate set $A_6(a_5)\subset
  N_{X_6}(a_3,a_5)$, which again guarantees the correct adjacencies
  of~$a_6$ to vertices with smaller indices. In addition, $a_6$ should be
  adjacent jointly with~$a_5$ to~$a_7$ and $(a_5,a_6)$ should be
  $(\frac12,p)$-connected to~$S$. So for each $a_5\in A_5$
  let $A_6(a_5)\subset N_{X_6}(a_3,a_5)$ be those vertices
  with at least $p^2|S|/80$ neighbours in $N_{X_7}(a_5)$, at least
  $p|S|/2$ neighbours in $S$, and at least $p^2|S|/4$ neighbours in
  $N_S(a_5)$. Again, by the definition of~$A_3$ and~$A_5$ and
  Lemma~\ref{lem:pseudorandom} we have
 \begin{equation*}
    |A_6(a_5)|
    \ge |N_{X_6}(a_3,a_5)|-3\eps p^2 n
    \ge p^2|S|/80-3\eps p^2 n\ge p^2|S|/160\,.
  \end{equation*}
  
  Similarly, we have to guarantee that candidates for~$a_7$ that are adjacent
  to~$a_5$ have a valid choice for a neighbour~$a_6$. Moreover, $a_7$
  should be adjacent jointly with~$a_5$ to~$a_8$, and $(a_7,a_8)$ should be
  $(\frac12,p)$-connected to~$S$. In addition we will guarantee that
  $(a_7,a_8)$ is also $(\frac12,p)$-connected to~$R^*$. We do not actually
  need this for $(a_7,a_8)$, but we will need it for $(b_7,b_8)$; and since
  we want to construct the candidate sets for the subgraph on
  $b_3,\dots,b_8$ analogously we will require it here.  So for each $a_5\in
  A_5$, let $A_7(a_5)\subset N_{X_7}(a_5)$ be those vertices with a
  neighbour in $A_6(a_5)$, with at least $p^2|S|/80$ neighbours in
  $N_{X_8}(a_5)$, with at least $p|S|/2$ neighbours in~$S$, and at least
  $p|R^*|/2$ neighbours in~$R^*$. 
  Because $|A_6(a_5)|\ge p^2|S|/160$ and $|N_{X_8}(a_5)|\ge p|S|/40$ by the
  definition of~$A_5$, Lemma~\ref{lem:pseudorandom} implies again
 \begin{equation*}
    |A_7(a_5)|
    \ge |N_{X_7}(a_5)|-\eps p n-3\eps p^2 n
    \ge p|S|/40-\eps p n-3\eps p^2 n\ge p|S|/80\,.
  \end{equation*}
  Set $A_7:=\bigcup_{a_5\in A_5}A_7(a_5)$. Then each
  vertex $a_5\in A_5$ has at least $|A_7(a_5)|\ge p|S|/80$ neighbours in $A_7$, and so
  by~\eqref{eq:pseudodisc} we have
  \[|A_5|p|S|/80\le e(A_5,A_7)\le (1+\eps)p|A_5||A_7|\,,\]
  and thus
  \begin{equation}\label{eq:res:Y7}
    |A_7|\ge |S|/160\,.
  \end{equation}
  So we have a candidate set for~$a_7$ which is of linear size. This will
  be crucial for the pigeonhole argument below.

  Finally we can turn to the definition of candidate sets for~$a_8$, which
  needs to be adjacent to~$a_5$ and~$a_7$ and such that $(a_7,a_8)$ is
  well-connected to~$S$ and~$R^*$. So for each
  $a_5\in A_5$ and $a_7\in A_7(a_5)$, let $A_8(a_5,a_7)\subset
  N_{X_8}(a_5,a_7)$ be those vertices with at least $p|S|/2$ neighbours in $S$, at least $p|R^*|/2$
  neighbours in $R^*$, at least $p^2|S|/4$ common neighbours with $a_7$ in $S$ and at least
  $p^2|R^*|/4$ common neighbours with $a_7$ in $R^*$. By
  Lemma~\ref{lem:pseudorandom} and the definition of~$A_5$ and $A_7(a_5)$ we have
  \begin{equation}\label{eq:res:Y8}
    |A_8(a_5,a_7)|
    \ge |N_{X_8}(a_5,a_7)|-4\eps p^2 n
    \ge p^2|S|/80-4\eps p^2 n\ge p^2|S|/160\,.
  \end{equation}
  This defines all the candidate sets for the vertices~$a_i$.

  By the symmetry of the $a_i$ and $b_i$ we can carry out the same construction 
  in the sets $X'_3,\ldots,X'_8$ to obtain $b_3$ and
  candidate sets $B_4$, $B_5$, $B_6(b_5)$, $B_7(b_5)$, $B_7$, $B_8(b_5,b_7)$
  for $b_4,\ldots,b_8$. 

  For completing the spine it now remains to use these candidate sets to
  find an edge $a_7b_7$ such that we can choose valid vertices
  $a_3,\dots,a_6,a_8$ and $b_3,\dots,b_6,b_8$ in the respective candidate
  sets and moreover $(a_8,a_7,b_7,b_8)$ forms a $2$-path. Thus we
  would in particular like to require that the vertex~$b_7$ in this edge
  has some neighbour in a candidate set for~$a_8$ defined
  by~$a_7$. However, the candidate sets for~$a_8$ depend in addition on
  vertices $a_5\in A_5$. This motivates the following definition.  We call
  an edge $a_7b_7$ with $a_7\in A_7$ and $b_7\in B_7$ \emph{good with
    respect to $a_1a_2$} if $b_7$ has a neighbour in $A_8(a_5,a_7)$ for
  some $a_5\in A_5$ with $a_7\in A_7(a_5)$.  Observe that by the definition
  of~$A_7$ for each $a_7\in A_7$, there indeed exists such an~$a_5$.
  Similarly, $b_7a_7$ is \emph{good with respect to $b_1b_2$} if $a_7$ has
  a neighbour in $B_8(b_5,b_7)$ for some $b_5\in B_5$ with $b_7\in
  B_7(b_5)$.

  Observe now that for each $a_5\in A_5$ and $a_7\in A_7(a_5)$
  by~\eqref{eq:res:Y8} and Lemma~\ref{lem:pseudorandom} at most $\eps pn$
  vertices in~$B_7$ have no neighbour in $A_8(a_5,a_7)$. Hence for all but
  at most $\eps p n$ vertices~$b_7$ in~$B_7$ the edge $a_7b_7$ is good with
  respect to $a_1a_2$. Similarly for all but at most $\eps pn$ vertices
  $a_7$ in $A_7$ the edge $b_7a_7$ is good with respect to $b_1b_2$.
  By~\eqref{eq:pseudodisc}
  there are at least $(1-\eps)p|A_7||B_7|$ edges between $A_7$ and $B_7$, of
  which all but at most $\eps p n\big(|A_7|+|B_7|\big)$ are good with
  respect to both $a_1a_2$ and $b_1b_2$. Since
  \[(1-\eps)p|A_7||B_7|\geByRef{eq:res:Y7}p|S|^2/51200>\eps p n^2\ge\eps p
  n\big(|A_7|+|B_7|\big)\] we can choose such a good edge $a_7b_7$.

  We will now complete this good edge and the already chosen $a_3,b_3$ to a
  copy of the spine. Since $a_7b_7$ is good with respect to $a_1a_2$ there
  exists a vertex $a_5\in A_5$ such that $a_7\in A_7(a_5)$ and $b_7$ has a
  neighbour~$a_8$ in $A_8(a_5,a_7)$. Fix such vertices~$a_5$
  and~$a_8$. Similarly, using goodness with respect to $b_1b_2$, we fix
  $b_5$ and $b_8$. This readily implies by the definition of the candidate
  sets that $(a_8,a_7,b_8,b_7))$ forms a
  $2$-path and that~$a_5$ is adjacent to~$a_7$ and~$a_8$ (and to~$a_3$),
  and similarly for~$b_5$.
  Next, by the definition of $A_7(a_5)$ we have that~$a_7$ has a
  neighbour~$a_6$ in $A_6(a_5)$, which we fix. Again, this implies
  that~$a_6$ is adjacent to~$a_5$ and~$a_7$. Moreover, by the definition
  of~$A_5$ we have that~$a_5$ has a neighbour~$a_4$ in~$A_4$, which we fix
  and which is thus adjacent to~$a_5$ and by the definition of~$A_4$
  to~$a_3$ and~$a_2$. Similarly, we fix~$b_6$ and~$b_4$.

  This completes the construction of the spine on
  $a_1,\dots,a_8,b_1,\dots,b_8$. In this construction we have guaranteed
  that $(a_1,a_2)$ and $(b_7,b_8)$ are $(\frac12,p)$-connected to~$R$
  and~$S$ as required and that the pairs from~\eqref{eq:res:conn} are
  $(\frac12,p)$-connected to~$S$.
  
  In our final step, which will complete the reservoir graph, we will apply
  Lemma~\ref{lem:connect} three times, to connect $(b_2,b_1)$ to $(a_4,a_3)$,
  $(a_5,a_6)$ to $(b_4,b_3)$, and $(b_5,b_6)$ to $(a_8,a_7)$.  More precisely let
  $S'=S\setminus\{a_1,\ldots,a_8,b_1,\ldots,b_8\}$. By
  Remark~\ref{rem:pseudo} we have $|S'|\ge \delta n/2-16\ge\delta n/4$. We
  apply Lemma~\ref{lem:connect} with $k=2$ and $\delta/4$ to find a
  ten-vertex squared path $P_1$ joining $b_2b_1$ to $a_4a_3$ in $S'$, then
  again to find $P_2$ joining $a_5a_6$ to $b_4b_3$ in $S'\setminus P_1$,
  and once more to find $P_3$ joining $b_5b_6$ to $a_8a_7$ in $S'\setminus
  (P_1\cup P_2)$. This is possible since all these pairs are
  $(\frac12,p)$-connected to~$S$, and, since $p^2|S|\ge 100$, even after
  removing the at most $47$ vertices of the partially constructed reservoir
  graph we have $(\frac14,p)$-connectedness.
 
  To conclude we have that
  \[a_1a_2rb_2b_1P_1a_4a_3a_5a_6P_2b_4b_3b_5b_6P_3a_8a_7b_7b_8\] is a
  squared path on $47$ vertices from $a_1a_2$ to $b_7b_8$ using $r$. On the
  other hand, by taking each $P_i$ in the opposite direction to the
  previous squared path we obtain that
  \[a_1a_2a_3a_4P_1b_1b_2b_3b_4P_2a_6a_5a_7a_8P_3b_6b_5b_7b_8\] is a
  squared path from $a_1a_2$ to $b_7b_8$ which uses every vertex of the
  previous squared path except $r$. Hence we have constructed the desired
  $2$-reservoir graph with reservoir vertex $r$.
\end{proof}

\section{Enumerating powers of Hamilton cycles}
\label{sec:count}
To prove Theorem~\ref{thm:countHC} we would ideally like to show that we
can construct the $k$th power of a Hamilton cycle vertex by vertex, and
that when we have $t$ vertices remaining uncovered, we have at least
$(1-\nu)p^k t$ choices for the next vertex; then the theorem would
follow immediately.  However, we obviously do not construct $k$th powers of
Hamilton cycles in this way: we have very little control over choice in
constructing the reservoir paths and connecting paths. Moreover for the
promised number $(1-\nu)^np^{kn}n!$ of Hamilton cycles powers even the
Extension lemma, Lemma~\ref{lem:onestep}, does not provide the desired
number of choices in the greedy portion of the construction where we
do choose one vertex at a time. (We remark though that the proof of this
lemma, together with the rest of our proof does immediately provide us with
$c^n p^{(1-\nu)kn}\big((1-\nu)n\big)!$ Hamilton cycle powers for some
absolute constant $c>0$.)

Thus we have to upgrade the Extension lemma in two ways. Firstly, we have to
modify it to give us more choices in each step (after a few initial
steps). Secondly, it turns out that to obtain the desired number of
Hamilton cycles powers we have to apply the Extension lemma for longer,
that is, the leftover set will in the end only contain $\bigO\big(n/(\log n)^2\big)$
vertices. Thus we have to change the Extension lemma to deal with this
different situation.  This comes at the cost of slightly tightening the
pseudorandomness requirement.

It is not hard to check that such an upgrade is possible. In the lemma
below we will guarantee that for an end $k$-tuple $\vecx$ of a $k$-path
there are $\big(1-\frac\nu{2k})\deg_L(\vecx)$ valid extensions, where~$L$ is
the current set of leftover vertices. As we will argue below, this will
provide us with the right number of Hamilton cycle powers if we can
guarantee in addition that
$\deg_L(\vecx)\ge\big((1-\frac\nu{2k})p\big)^k|L|$. Recall however that we
will want to use this lemma after constructing the reservoir path with the
help of Lemma~\ref{lem:reslem}, which guarantees $(\frac18,p)$-connectedness
to~$L$, a property which only gives a weaker lower bound on $\deg_L(\vecx)$
than desired. In order to overcome this shortcoming we will in the first
few applications of the Counting version of the Extension lemma transform
this $(\frac18,p)$-connectedness to a stronger property which gives the
desired bound.  Conditions~\ref{better:L1} and~\ref{better:L2}, and
conclusions~\ref{better:L1'} and~\ref{better:L2'} take care of this.

\begin{lemma}[Counting version of the Extension lemma]\label{lem:betterstep}
  Given $k\ge 2$ and $\nu>0$, if $C=2^{k+23}k^4/\nu$ then the following
  holds.  Let $0<p<1$ and $G$ be an $\big(1/(C\log
  n)^2,p,k-1,k\big)$-pseudorandom graph on $n$ vertices. Let $L$ and $R$ be
  disjoint vertex sets with $|L|,|R|\ge n/(200k\log n)^2$. Suppose that
  there is $0\le j\le k$ such that $\vecx=(x_1,\ldots,x_k)$ satisfies
  \begin{enumerate}[label=\rom]
  \item  $\vecx$ is $\big(\tfrac{1}{8},p)$-connected to $R$,
  \item\label{better:L1} $\deg_L(x_i,\ldots,x_k)\ge \frac18 \big(\frac{p}{2}\big)^{k-i+1}|L|$
    for each $1\le i\le j$,
  \item\label{better:L2} $\deg_L(x_i,\ldots,x_k)\ge\big((1-\frac\nu{2k})p\big)^{k-i+1}|L|$ for each $j<i\le k$.
  \end{enumerate}
  Then at least $\big(1-\frac\nu{2k})\deg_L(x_1,\ldots,x_k)$
  vertices $x_{k+1}\in N_L(x_1,\ldots,x_k)$ satisfy that
  \begin{enumerate}[label=\abc]
  \item $(x_2,\ldots,x_{k+1})$ is $\big(\tfrac{1}{6},p)$-connected to $R$,
  \item\label{better:L1'} 
    $\deg_L(x_{i+1},\ldots,x_{k+1})\ge \frac18 \big(\frac{p}{2}\big)^{k-i+1}|L|$ for each $1\le i\le j-1$,
 \item\label{better:L2'}
    $\deg_L(x_{i+1},\ldots,x_{k+1})\ge\big((1-\frac\nu{2k})p\big)^{k-i+1}|L|$ for each $j-1<i\le k$.
  \end{enumerate}
\end{lemma}
\begin{proof}[Sketch of proof]
  In this proof we say that a vertex $x$ is \emph{typical} with
  respect to $S$ if $\deg_S(x)\ge \big(1-\frac\nu{2k}\big)p|S|$.
  Let $X_{k+1}\subset N_L(x_1,\ldots,x_k)$ be the set of all vertices
  typical with respect to $N_L(x_i,\ldots,x_k)$ for each $2\le i\le k$, to
  $R\cap N(x_i,\ldots,x_k)$ for each $2\le i\le k$, and to $L$ and $R$. By
  Lemma~\ref{lem:pseudorandom} the number of vertices which fail any one of
  these conditions is at most $4k n p^k/(C\log n)^2<\nu p^k |L|/(2k\cdot
  2^{-k-4})$.
\end{proof}

\begin{proof}[Sketch of proof of Theorem~\ref{thm:countHC}]
 We follow essentially the proof of Theorem~\ref{thm:main}. We make the
 following alterations. We require the same pseudorandomness of $G$ as in that
 theorem, \emph{except} that we set $C=2^{k+23}k^4/\nu$ and choose $\eps\le
 1/(C\log n)^2$ such that any $n/(\log n)^2$-vertex induced subgraph of $G$ still meets the pseudorandomness
 requirements of Theorem~\ref{thm:main}. By equation~\eqref{eq:pseudodisc}, this
 means we can choose $\eps=\Theta\big(1/(\log n)^2\big)$.
 
 We construct a reservoir set $R$ of size between $n/(\log n)^2$ and $10n/(\log n)^2$, with the
 properties that all vertices of $R$ have at least $\beta p n/2$ neighbours
 in $V(G)\setminus R$ and all vertices of $V(G)\setminus R$ have at least
 $\frac12\beta p |R|$ neighbours in $R$. It is not hard to check that the
 same construction procedure works.
 
 The application of Lemma~\ref{lem:reslem} to obtain a reservoir path
 $P$ of size $50k|R|$ covering $R$ needs to change only in that we have to guarantee
 $(\tfrac12,p)$-connection of the ends of $P$ to $R$.
 The choice of $\eps$ obviously allows this.
 
 We now use Lemma~\ref{lem:betterstep} instead of
 Lemma~\ref{lem:onestep} to extend $P$ greedily. In the first $k$ steps
 we let the input $j$ decrease from $k$ to $1$, while at the $(k+1)$st
 application and thereafter, we take $j=0$. We continue until we can no longer use
 Lemma~\ref{lem:betterstep}, i.e.\ at the step when we have
 constructed $P'$ and the number of vertices not in $R\cup P'$ is less
 than $n/(200k\log n)^2$. We set $L$ to be these leftover vertices.
 
 The remainder of the proof is identical to that of Theorem~\ref{thm:main}: that
 is, we apply Lemma~\ref{lem:covlem} to construct a path $P''$ covering $L$ and
 using only vertices from $L\cup R$, and connect $P''$ and $P'$ using
 Lemma~\ref{lem:connect}. The choice of $\eps$ ensures that we can do this since
 $|L\cup R|>n/(\log n)^2$. We thus successfully construct the $k$th power of a
 Hamilton cycle in $G$.
 
 Finally, by considering the choices only during the use of
 Lemma~\ref{lem:betterstep} with $j=0$, we can estimate the number of
 $k$th powers of cycles which we construct are at least
 \begin{multline*}
   \prod_{t=n/(200k\log n)^2}^{n-1000kn/(\log n)^2}
    \big(1-\tfrac{\nu}{2k}\big)^kp^k t
    \ge \prod_{t=n/(200k\log n)^2}^{n-1000kn/(\log n)^2}
    \big(1-\tfrac{\nu}{2}\big)p^k t \\
    \ge \big(1-\tfrac{\nu}{2}\big)^np^{kn}n!n^{-1001kn/(\log n)^2}\,.
  \end{multline*}
 and since $n^{-1001kn/(\log n)^2}=\big(2^{-1001k/\log
 n}\big)^n>\big(1-\tfrac{\nu}{2}\big)^n$ for sufficiently large $n$, the
 result follows.
\end{proof}

\section{Concluding remarks}
\label{sec:conc}

\paragraph{\bf Hamilton cycles}
For Hamilton cycles, a simple modification of our arguments for squared Hamilton cycles yields that
$\big(\eps,p,0,1\big)$-pseudorandom graphs with minimum degree $\beta p
n$ are Hamiltonian for sufficiently small $\eps=\eps(\beta)$. This bound is
essentially best possible (for our notion of pseudorandomness) since the
disjoint union of $G(n-pn,p)$ and $K_{pn}$ is easily seen to be 
asymptotically almost surely
$(\eps,p,0,1-\eps)$-pseudorandom and have minimum degree at least $pn/2$. 

\medskip

\paragraph{\bf Improving the pseudorandomness requirements}
It would be interesting to obtain stronger results on the pseudorandomness
required to find $k$th powers of Hamilton cycles. We believe that a
generalisation of our result for the $k=2$ case is true.

\begin{conjecture}
  For all $k\ge 2$ the pseudorandomness requirement in
  Theorem~\ref{thm:main} can be replaced by
  $(\eps,p,k-1,k)$-pseudorandomness.
\end{conjecture}

As remarked in the introduction even in the $k=2$ case we do not know
whether Theorem~\ref{thm:main} is sharp.  It would also be very interesting
(albeit very hard) to find better lower bound examples than those
mentioned in the introduction.

In the evolution of random graphs triangles, spanning triangle factors and
squares of Hamilton cycles appear at different times: In $G(n,p)$ the
threshold for triangles is $p=n^{-1}$, but only at $p=\Theta(n^{-2/3}(\log
n)^{1/3})$ each vertex of $G(n,p)$ is contained in a triangle with high
probability, which is also the threshold for the appearance of a spanning
triangle factor~\cite{JKV08}. Squares of Hamilton cycles on the other hand
are with high probability not present in $G(n,p)$ for $p\le n^{-1/2}$, and
K\"uhn and Osthus~\cite{KOPosa} recently showed that for $p\ge n^{-1/2+\eps}$
they are present. Our Theorem~\ref{thm:main} is also applicable to random graphs, 
but the range is worse: $p\gg (\ln n/n)^{1/3} $ for squares of Hamilton cycles 
and $p\gg (\ln n/n)^{1/(2k)}$ for general $k$th powers of Hamilton
cycles (recall that Riordan's result~\cite{Riordan} implies the
optimal bound $p\gg n^{-1/k}$ for $k\ge 3$). 

Pseudorandom graphs behave differently. For $(n,d,\lambda)$-graphs it is
known that there are triangle-free $(n,d,\lambda)$-graphs with 
$\lambda=c d^2/n$ for some $c$, but for `small'~$c$  
every vertex in an $(n,d,\lambda)$-graph is contained in a triangle
(and, more generally, there exists a fractional triangle factor). This
motivated Krivelevich, Sudakov and Szab\'o~\cite{KSS04} to conjecture
that indeed these graphs already contain a spanning triangle
factor. We do not know whether triangle factors and squares of
Hamilton cycles require differently strong pseudorandomness
conditions.

\begin{question}
  Do spanning triangle factors and spanning $2$-cycles appear for the same
  pseudorandomness requirements (up to constant factors)?
\end{question}

\medskip

\paragraph{\bf Universality}
For random graphs the study of when $G(n,p)$ contains all spanning or
almost spanning graphs with maximum degree bounded by a constant~$\Delta$,
was initiated in~\cite{ACKRR_universality}. In this case $G(n,p)$ is also
called universal for these graphs.  The authors
of~\cite{ACKRR_universality} showed that $G(n,p)$ contains all graphs on
$(1-\eps)n$ vertices with maximum degree at most~$\Delta$ if $p\ge
Cn^{-1/\Delta}\log^{1/\Delta}n$. In~\cite{DKRR_improved} this result was
extended to such subgraphs on~$n$ vertices. Recently,
Conlon~\cite{Con_universality} announced that for the first of these two
results he can lower the probability to $p=n^{-\eps-1/\Delta}$ for some
(small) $\eps>0$.  The best known lower bound results from the fact
that $p=\Omega(n^{-2/(\Delta+1)})$ is necessary for $G(n,p)$ to contain a
$K_{\Delta+1}$-factor.

For pseudorandom graphs we were only recently able to establish
universality results of this type, which follow from our work on a Blow-up lemma for
pseudorandom graphs (see below). We can prove that
$(p,cp^{\frac32\Delta+\frac12}n)$-jumbled graphs on~$n$ vertices with minimum
degree $\beta p n$ are universal for spanning graphs with maximum
degree~$\Delta$~\cite{ABHKP_sparse_blowup}. We believe that these
conditions are not optimal.

\begin{question}
  Which pseudorandomness conditions (plus minimum degree conditions) imply
  universality for spanning graphs of maximum degree~$\Delta$?
\end{question}

It is worth noting that Alon and Capalbo~\cite{AlCap} explicitly
constructed almost optimally sparse universal graphs for spanning graphs
with maximum degree~$\Delta$. These graphs have some pseudorandomness
properties, but they also contain cliques of order $\log^2 n$, which random
graphs of the same density certainly do not.

%
%

\medskip

\paragraph{\bf Additive structures in multiplicative subgroups} Alon and
Bourgain~\cite{AlonBourgain} recently made use of properties of
pseudorandom graphs in order to prove the following conjecture of
Sun~\cite{Sun}. Given any prime $p\ge 13$, there is a cyclic ordering of
the quadratic residues modulo $p$ such that the sum of any two
consecutive quadratic residues in the cyclic order is also a quadratic
residue. In fact, Alon and Bourgain proved much more. It is not
necessary to take the
subgroup of $\FF_p$ formed by the quadratic residues. Any sufficiently
large multiplicative subgroup of $\FF_p$ has the same property.

\begin{theorem}[Alon and Bourgain~\cite{AlonBourgain}, Theorem 1.2]
  There exists an absolute positive constant $c$ such that for any
  prime power $q$ and for any multiplicative subgroup~$A$ of the
  finite field~$\FF_q$ of size
  \[|A|=d\ge c\,\frac{q^{3/4}(\log q)^{1/2}(\log\log\log
    q)^{1/2}}{\log\log q}\] there is a cyclic ordering
  $a_0,a_1,\ldots,a_{d-1}$ of the elements of $A$ such that
  $a_i+a_{i+1}\in A$ for all $i$.
\end{theorem}

The proof of this theorem amounts to showing that a certain graph on
vertex set~$A$ is pseudorandom and applying the Hamiltonicity result of
Krivelevich and Sudakov~\cite{KriSudHam} to find a Hamilton cycle in
this graph, which defines the cyclic ordering. We can replace that
result with Corollary~\ref{cor:ndlambda} to obtain the following result,
which in particular strengthens Proposition~1.6 of~\cite{AlonBourgain}.

\begin{corollary}\label{cor:groups}
  For each $k\ge 2$ there exists a positive constant $c$ such that for
  any prime power $q$ and multiplicative subgroup $A$ of $\FF_q$ of
  size
  \begin{equation*}
    |A|=d\ge\begin{cases} c q^{6/7} & k=2\\ c q^{(3k+1)/(3k+2)} & k\ge
      3\end{cases}
  \end{equation*}
  there is a cyclic ordering $a_0,a_1,\ldots,a_{d-1}$ of the
  elements of $A$ such that $a_i+a_{i+j}$ is in $A$ for all $i$ and $1\le
  j\le k$.
\end{corollary}
%

\medskip

\paragraph{\bf Blow-up lemmas}

For dense graphs the Blow-up lemma~\cite{KSS_bl} is a powerful tool for
embedding spanning graphs with bounded maximum degree (versions of this
lemma for certain graphs with a maximum degree not bounded by a constant have
recently been developed in~\cite{BTWBlowup}).
Already Krivelevich, Sudakov and Szab\'o~\cite{KSS04} remark that
their result on triangle factors in sparse pseudorandom graphs 
can be viewed as a first step towards the development of a Blow-up lemma for
(subgraphs of) sparse pseudorandom graphs.

We see the results presented here as a further step in this direction.  And
in fact in recent work~\cite{ABHKP_sparse_blowup} we establish a
blow-up lemma for spanning graphs with bounded maximum degree in sparse
pseudorandom graphs.  However, the pseudorandomness requirements for this
more general result are more restrictive than those used here.


\bibliographystyle{amsplain_yk}
\bibliography{PseudoHCk}

\end{document}